\documentclass[11pt]{article}

\usepackage{style}

\newtheorem{Theorem}{Theorem}
\numberwithin{Theorem}{section}

\newtheorem{Proposition}[Theorem]{Proposition}
\newtheorem{Lemma}[Theorem]{Lemma}
\newtheorem{Remark}[Theorem]{Remark}
\newtheorem{Definition}[Theorem]{Definition}
\newtheorem{Example}[Theorem]{Example}
\newtheorem{corollary}[Theorem]{Corollary}
\newtheorem*{Ack}{Acknowledgements}

\newtheorem*{theorem}{Theorem}

\title{Quiver representations and idempotent loops.}
\author{Liam Riordan}
\date{}

\begin{document}

\maketitle

\begin{abstract}
    We study $\cstar$-actions on quiver representations by adding idempotent loops. We construct a category $\Mod \limmauro{Q}$ and relate it to $\mathbb{C}Q$ modules with $\cstar$-actions. We further describe a subcategory $\tlimmauro{Q}{T}$ equivalent to the category of equivariant modules where $\cstar$ acts on $\mathbb{C}Q$ via a character $T \in \mathbb{Z}^{Q_1}.$ This allows us to study $\cstar$-actions on quiver Grassmannians and we can recover known combinatorial descriptions of their Euler characteristics via the category $\Mod \limmauro{Q}.$ We directly relate morphisms of quivers to full subcategories of $\Mod \limmauro{Q}.$ Finally we show that the representation theory of quivers with idempotents can be applied in a useful way to preprojective algebras. We show that this gives constructions of Galois covers and cluster characters used by Geiss, Leclerc, and Schr\"oer.
\end{abstract}

\tableofcontents

\section{Introduction}
In this paper we will consider how to encode $\cstar$-actions on $\crep$ via representation theory. The combinatorics of such actions has already been considered and used to study Euler characteristics of quiver Grassmannians by Cerulli-Irelli \cite{CI} and Haupt \cite{HAU}. Our construction will be motivated by the following problem. Consider a representation $M \in \crep$ and $Q^0$ a quiver with the same vertices as $Q$ and with all of the arrows of $Q$ plus a loop at each vertex. See section \ref{section 2} for more details on the notation. Denote by $i_0(M)$ the representation of $Q^0$ where we have simply added the identity maps to each of the loops and all other data is determined by M. Observe that \[\gr_{\textbf{d}}(i_0(M)) \cong \gr_{\textbf{d}}(M)\] and given $N \in \crep$ we have \[\Hom_Q(M,N) \cong \Hom_{Q^0}(i_0(M),i_0(N)).\] See \ref{Hom fix} and \ref{Grass fix} for details of these isomorphisms. It is natural to ask the following question: What happens to the above isomorphisms when we break the identity loops into several loops giving an orthogonal decomposition of the identity idempotent. In this paper we will study the representations which arise from this process.  

For a representation $M \in \crep$ of dimension vector $\textbf{m}$ consider a morphism of algebraic groups \[t: \cstar \rightarrow \operatorname{G}_{\textbf{m}} \defeq \displaystyle\prod_{x \in Q_0} \mathrm{GL}(\mathbb{C}^{m_x}).\] Such a $t$ gives rise to a $\cstar$-action on the representation variety \[E_{\textbf{m}} \defeq \displaystyle\bigoplus_{a \in Q_1} \Hom(M_{s(a)},M_{t(a)})\] as well as on \[M_{\textbf{m}} \defeq \displaystyle\bigoplus_{x \in Q_0} M_x.\] Here $M_x$ for $x \in Q_0$ denotes the vector space at vertex $x$ of $M.$ Given a representation $M \in \crep$ of dimension vector $\textbf{m}$ we will define a representation $M(t)$ by adding loops at each vertex which are idempotents projecting onto the weight spaces given by the $\cstar$-action on $M_{\textbf{m}}.$ We will show that $M(t)$ encodes the representation theory of $M$ along with the $\cstar$-action.

Given $M(t)$ the quiver Grassmannian $\gr_{\textbf{d}}(M(t))$ is a subvariety of the quiver Grassmannian of the underlying quiver representation $M.$ It is natural to ask how different choices of idempotents change the subvariety $\gr_{\textbf{d}}(M(t)) \subseteq \gr_{\textbf{d}}(M).$ For an action of $\cstar$ on $X$ we denote by $X^{\cstar}$ the subset of fixed points under this action. We show that 
\[\gr_{\textbf{d}}(M(t)) = \gr_{\textbf{d}}(M) \cap \displaystyle\prod_{a \in Q_0} \gr (d_a, m_a)^{\mathbb{C}^{\times}}.\]
Given $M(t)$ such that $\gr_{\textbf{d}}(M)$ is $\cstar$-stable then $\gr_{\textbf{d}}(M(t))$ is a torus fixed point subvariety of $\gr_{\textbf{d}}(M).$ This would allow us to compute Euler characteristics for quiver Grassmannians. Such computations play a role in cluster theory where Euler characteristics of quiver Grassmannians are coefficients of cluster monomials in terms of a fixed seed \cite{FK}. To formalise the representations $M(t)$ we define algebras $\mauro{j}{Q}$ for $j \in \mathbb{N}$ whose module categories consists of $\mathbb{C}Q$ modules with idempotent loops. These come with certain quotient algebras $\tmauro{j}{Q}{T}$ for $T \in \mathbb{Z}^{Q_1}$ which correspond to equivariant representations. We construct a colimit of the categories $\Mod \mauro{j}{Q}$ for $j \in \mathbb{N}$ which can be thought of as a category modelling quiver representations with certain $\mathbb{C}^{\times}$-actions. There is a functor $R\colon \Mod \mauro{j}{Q} \rightarrow \crep$ which sends a $\mauro{j}{Q}$ to its underlying $Q$-representation. We show that the colimit of the categories $\Mod \mauro{j}{Q}$ exists and denote it $\Mod \limmauro{Q}.$ Likewise, we construct a colimit $\Mod \tlimmauro{Q}{T}.$ 

We describe the cases where $\gr_{\textbf{d}}(M)$ is stable under the $\cstar$-action using the category $\Mod \tlimmauro{Q}{T},$ which we show is equivalent to equivariance (Theorem \ref{equivariant}). We further explain how these techniques give an algebraic framework (Proposition \ref{haupt}) for combinatorial results of Cerulli-Irelli \cite{CI} and Haupt \cite{HAU} for computing Euler characteristics. In particular, we define the notion of a limit (Definition \ref{limit}) of an equivariant sequence. 

\begin{theorem}[Theorem \ref{limit euler}]
    Given $X$ a limit of an equivariant sequence $(\limnot{0}{M},\cdots,\limnot{r}{M})$ we have \[\chi(\gr_\textbf{d}(X)) = \chi(\gr_\textbf{d}(R(X))), \, \forall \textbf{d} \in \mathbb{N}^{Q_0}.\] 
\end{theorem}

We further show (Proposition \ref{limit to haupt}) that equivariant sequences are equivalent to nice sequences in the sense of Haupt \cite{HAU}. Given an object $N \in \Mod \limmauro{Q}$ then there is a map $t_N: \cstar \rightarrow \mathrm{G}_\textbf{n}.$ Given another object $M \in \Mod \limmauro{Q}$ we denote by $\Hom(\textbf{d},M) \subseteq \oplus_x \Hom(\mathbb{C}^{d_x},M_x)$ the space of maps $(f_x)_{x \in Q_0}$ which give rise to a morphism $(f_x)_{x\in Q_0}:U \rightarrow M$ for some $U \in \Mod \limmauro{Q}.$ If we have an equality of dimension vectors $\textbf{m} = \textbf{n}$ then by fixing an isomorphism of the vector spaces $N_x \cong M_x$ for each $x \in Q_0$ there is an action of $\mathrm{G}_\textbf{n}$ on $\Hom(\textbf{d},M)$ for any $\textbf{d} \in \mathbb{N}.$ We say that $\Hom(-,M)$ is $t_N$-stable if the subspace $\Hom(\textbf{d},M)$ is $t_N$-stable for all $\textbf{d} \in \mathbb{N}^{Q_0}.$

\begin{theorem}[Theorem \ref{Main theorem}]
    Let $(\limnot{0}{M},\limnot{1}{M},\cdots,\limnot{r}{M})$ be a collection of objects in $\Mod \limmauro{Q}$ such that for each i we have 
    \begin{itemize}
        \item $R(\limnot{i}{M}) = R(\limnot{i-1}{M}),$
        \item $\Hom(-,\limnot{i-1}{M})$ is $t_{\limnot{i}{M}}$-stable,
        \item $\Hom(\textbf{d},\limnot{i}{M}) \subseteq \Hom(\textbf{d},\limnot{i-1}{M})$ $\forall \textbf{d} \in \mathbb{N}^{Q_0}.$ 
    \end{itemize}
    Under these conditions 
    \[\chi(\gr_\textbf{d}(\limnot{r}{M})) = \chi(\gr_\textbf{d}(\limnot{0}{M})).\]    
\end{theorem}

Typically the quiver Grassmannian on the left will be of strictly smaller dimension than the quiver Grassmannian on the right. In special cases the quiver Grassmannian on the left is zero dimensional and so computing the Euler characteristic becomes a point count. 

We relate the algebraic construction to the prior combinatorial constructions. In particular we relate representations of $\mauro{j}{Q}$ to representations of quivers $S$ along with morphisms of quivers $F:S\rightarrow Q.$ A morphism of quivers $F:S\rightarrow Q$ along with a map $v:S_0\rightarrow \mathbb{Z}$ satisfying certain extra properties gives rise to an embedding of categories \[\rho_{F,v}: \mathrm{Rep}(S,\mathbb{C}) \hookrightarrow \Mod \limmauro{Q}.\] Finally we apply the framework of idempotent loops to preprojective algebras to describe Galois coverings as well as a space of constructible functions. 

\begin{Ack}
    The author was supported by the Additional Funding Programme for Mathematical Sciences via UKRI/EPSRC (EP/V521917/1) and the Heilbronn institute. We would like to thank Alastair King for suggesting a very useful improvement to the presentation of results in this paper. We also thank Ellis Caird for reading this paper and helping with the rewording of many statements. We would like to thank Jan Grabowski and Lewis Topley for giving many helpful corrections which made this paper much easer to read.
\end{Ack}

\section{Representations and actions}\label{section 2}
Given a quiver, $Q = (Q_0, Q_1, s, t)$ the data of a (finite dimensional $\mathbb{C}$) representation of $Q$ will consist of a $\mathbb{C}$ vector space $M_x$ for each $x \in M_x$ and a $\mathbb{C}$-linear map $M_a: M_{s(a)} \rightarrow M_{t(a)}$ for each $a \in Q_1.$ For $x \in Q_0$ and $i \in \mathbb{Z}$ we will denote by $\ell_i^x$ a loop at vertex $x.$ The role of the index $i$ will become important later. Given a quiver $Q$ we denote by $Q^j$ the quiver with $Q^j_0 = Q_0$ and \[Q^j_1 = Q_1 \cup \{\ell_i^x : x \in Q_0, \, -j \leq i \leq j\}.\]

\begin{Example}
    For $Q$ the quiver \[\begin{tikzcd}
	{x} & {y}
	\arrow[from=1-2, to=1-1] 
    \end{tikzcd}\] the quiver $Q^1$ is given by 
    \[\begin{tikzcd}[ampersand replacement=\&]
	{x} \& {y}
	\arrow[loop above, from=1-1, to=1-1, "{\ell_{-1}^x}"]
    \arrow[loop left, from=1-1, to=1-1, "{\ell_{0}^x}"]
    \arrow[loop below, from=1-1, to=1-1, "{\ell_{1}^x}"]
	\arrow[from=1-2, to=1-1]
	\arrow[loop above, from=1-2, to=1-2, "{\ell_{-1}^y}"]
    \arrow[loop right, from=1-2, to=1-2, "{\ell_{0}^y}"]
    \arrow[loop below, from=1-2, to=1-2, "{\ell_{1}^y}"]
    \end{tikzcd}\]  
\end{Example}

\begin{Definition} \label{Maurolico algebra}
    Given a quiver $Q$ and $j \in \mathbb{N}$ we associate to it the algebra
    \[\mauro {j} {Q} = \mathbb{C}Q^j/I\]
    where the ideal $I$ is generated by the following:
    \begin{itemize}
        \item at each vertex $x \in Q_0$ we impose $\ell^{x}_i \ell^{x}_j$ if $i \neq j,$
        \item at each vertex $x \in Q_0$ we impose $\displaystyle\sum_{i = -j}^{j} \ell^{x}_i - e_x$ where $e_x$ is the trivial path of length zero at vertex $x.$
    \end{itemize}

    Given $T \in \mathbb{Z}^{Q_1}$ we define $\tmauro{j} {Q} {T}$ as the quotient \[\tmauro{j} {Q} {T} = \mauro {j} {Q}/I_T\] 
    where $I_T$ is generated by imposing \[\ell^{t(a)}_{i} a \ell^{s(a)}_k\] when $i - k \neq T_a$ for each arrow $a \in Q_1.$ We refer to this condition as \textbf{T-homogeneity}. It follows that $\Mod \tmauro{j} {Q} {T}$ is a full subcategory of $\Mod \mauro {j} {Q}$ closed under submodules and quotients.
\end{Definition}  

We will repeatedly make use the following well known result.

\begin{Proposition}[\protect{\cite[Section XI]{BoS}}]\label{restriction}
  Given a morphism of algebras $f: A \rightarrow B$ then there is a functor $\Mod B \rightarrow \Mod A$ called restriction. If the morphism $f$ is surjective then this is a fully faithful functor, i.e. an embedding. If $f: A \rightarrow B$ is injective then the restriction functor is faithful.  
\end{Proposition}

\begin{Remark}
    The $T$-homogeneity condition says that for $M \in \Mod \tmauro{j}{Q}{T}$ and any $a \in Q_1$ \[M_a(\mathrm{Im}(M_{\ell_i^{s(a)}})) \subseteq \mathrm{Im}((M_{\ell_{i + T_a}^{t(a)}}).\] The maps $M_a$ send images of idempotents to images of idempotents where the index of the idempotent has been shifted by $T_a.$
\end{Remark}

Note that $Q$ is a subquiver of ${Q}^j$ for any $j \in \mathbb{N}.$ This gives rise to an injective algebra morphism $\mathbb{C}Q \hookrightarrow \mathbb{C}Q^j.$ If we compose this embedding with the quotient $\mathbb{C}Q^j \twoheadrightarrow \mauro{j}{Q} \cong \mathbb{C}Q/I$ then we get a morphism of algebras $\mathbb{C}Q \rightarrow \mauro{j}{Q}.$ 
\begin{Proposition}
    The morphism of algebras $\mathbb{C}Q \rightarrow \mauro{j}{Q}$ is injective.
\end{Proposition}
\begin{proof}
    It suffices to show that $\mathbb{C}Q \cap I = (0) \in \mathbb{C}Q^j$ for $I$ the ideal $I$ in $\mathbb{C}Q^j/I \cong \mauro{j}{Q}$ where we identify $\mathbb{C}Q$ as a subalgebra of $\mathbb{C}Q^j.$ This ideal is generated by $\ell_i^x \ell_j^x$ for $i \neq j$ and $e-\sum_{i=-j}^j \ell_i^x.$ Suppose we have \begin{equation}\label{idealeq}
    \sum_{i \neq j} a_{ij} \ell_i^x\ell_j^x + b\left(e_x - \sum_{i=-j}^j \ell_i^x\right) = c        
    \end{equation} for $a_{ij}, b \in \mathbb{C}Q^j$ and $c \in \mathbb{C}Q.$ We need to show $c = 0.$ We will write $v = \sum_{i=-j}^j \ell_i^x$ in order to help with notation. Consider the above equation as an equation in the polynomial ring $\mathbb{C}Q[\ell_i^x, i \in \{-j, \cdots, j\}].$ That is to say polynomials in the variables $\ell_i^x$ with coefficients in $\mathbb{C}Q.$ This allows us to consider degrees of monomials. The polynomial ring $\mathbb{C}Q[\ell_i^x, i \in \{-j, \cdots, j\}]$ is naturally $\mathbb{N}$-graded by total degrees of monomials. The equation \ref{idealeq} gives us relations in each graded piece which we now consider. We will write $a = \sum_{i \neq j} a_{ij} \ell_i^x\ell_j^x$ and rearrange to consider $a = c - b(e_x - v)$ giving  
    \begin{equation}\label{gradeq}
    \begin{split}
    0 &= c - b_0, \\
    0 &= b_1 - b_0v,\\
    a_2 &= b_2 - b_1v,\\
    &\cdots \\
    a_{r-1} &= b_{r-1} - b_{r-2}v, \\
    a_r &= -b_{r-1}v     
    \end{split}
    \end{equation} 
    where $r$ is the maximal degree occurring in the equation, $a_i$ is the degree $i$ component of a and $b_i$ is the degree $i$ component of b. We have used the fact that $a$ has no terms in degree 0 or 1 and the fact that the highest degree in $a$ must be given by the highest degree in $-bv.$ The degree zero part of equation \ref{gradeq} tells us $b_0 = c$ so it suffices to show $b_0 = 0.$ Now by the degree 1 relation we see $b_1 = b_0v.$ In particular the coefficient of $\ell_i^x$ in $b_1$ is $b_0.$ In degree 2 we see $a_2 = b_2 - b_1v.$ Since $a = a = \sum_{i \neq j} a_{ij} \ell_i^x\ell_j^x$ then the coefficient of $\left(\ell_i^x\right)^t$ is zero in degree t since any term contains a product of the form $\ell_i^x\ell_j^x$ for $i \neq j.$ In particular, since the coefficient of $\left(\ell_i^x\right)^2$ in $-b_1v = -b_0v^2$ is $-b_0$ then the coefficient of $\left(\ell_i^x\right)^2$ in $b_2$ is $b_0.$ We continually apply this argument inductively to see that from $a_{r-1} = b_{r-1} - b_{r-2}v$ the coefficient of $\left(\ell_i^x\right)^{r-1}$ in $b_{r-1}$ is a non-zero multiple of $b_0.$ Finally consider $a_r = -b_{r-1}v.$ The right hand side has a non-zero coefficient of $\left(\ell_i^x\right)^{r}$ if and only if $a_r$ does. But as described $a_r$ cannot have a non-zero coefficient for this term. This implies $b_0 = c = 0.$
\end{proof}

\begin{corollary}
    The algebra $\mathbb{C}Q$ is isomorphic to a subalgebra of $\mauro{j}{Q}$ for any $j\in \mathbb{N}.$ This gives rise to the restriction functor $R: \Mod \mauro{j}{Q} \rightarrow \Mod \mathbb{C}Q.$ By the above proposition and Proposition \ref{restriction} then this functor is faithful.
\end{corollary}

\begin{Lemma} \label{module cat}
    A $\mauro {j} {Q}$-module $N$ is a representation $R(N)$ of $Q$ and an orthogonal idempotent decomposition of the identity at each vertex. 
\end{Lemma}

\begin{proof}
    The restriction functor gives us the underlying representation $R(N)$. The second family of generators for the ideal $I$ says that the loops sum to the identity at that vertex. The first set of generators tell us the loops are orthogonal. They are idempotent since $e_x = \sum_i \ell^{x}_i$ therefore $\ell^{x}_i = \ell^{x}_ie_x = \sum_i \ell^{x}_i \ell^{x}_j = \ell^{x}_i \ell^{x}_i.$ 
\end{proof}

\begin{Example}
    Continuing our example with the quiver \[\begin{tikzcd}
    {x} & {y}
    \arrow[from=1-2, to=1-1] 
    \end{tikzcd}\] the algebra $\mauro{1}{Q}$ is given by 
    \[\begin{tikzcd}[ampersand replacement=\&]
    {x} \& {y}
    \arrow[loop above, from=1-1, to=1-1, "{\ell_{-1}^x}"]
    \arrow[loop left, from=1-1, to=1-1, "{\ell_{0}^x}"]
    \arrow[loop below, from=1-1, to=1-1, "{\ell_{1}^x}"]
    \arrow[from=1-2, to=1-1]
    \arrow[loop above, from=1-2, to=1-2, "{\ell_{-1}^y}"]
    \arrow[loop right, from=1-2, to=1-2, "{\ell_{0}^y}"]
    \arrow[loop below, from=1-2, to=1-2, "{\ell_{1}^y}"]
    \end{tikzcd}\] with $\ell_{-1}^x + \ell_0^x + \ell_{1}^x = e_x$ and $\ell_{i}^x\ell_{j}^x =0$ when $i\neq j$ at vertex $x,$ likewise at y. An example of an object $M \in \Mod \mauro{1}{Q}$ is 
    \[\begin{tikzcd}[ampersand replacement=\&]
    {\mathbb{C}} \& {\mathbb{C}^2}
    \arrow[loop above, from=1-1, to=1-1, "{1}"]
    \arrow[loop left, from=1-1, to=1-1, "{0}"]
    \arrow[loop below, from=1-1, to=1-1, "{0}"]
    \arrow[from=1-2, to=1-1, "{\begin{psmallmatrix} 1 & 1 \end{psmallmatrix}}"]
    \arrow[loop above, from=1-2, to=1-2, "{A}"]
    \arrow[loop right, from=1-2, to=1-2, "{B}"]
    \arrow[loop below, from=1-2, to=1-2, "{0}"]
    \end{tikzcd}\] where $A = \frac{1}{2} \begin{psmallmatrix} 1 & 1 \\ 1 & 1 \end{psmallmatrix}$ and $B = {\frac{1}{2} \begin{psmallmatrix} 1 & -1 \\ -1 & 1 \end{psmallmatrix}}.$ Here $R(M)$ is the $Q$ representation \begin{tikzcd}[ampersand replacement=\&]
    {\mathbb{C}} \& {\mathbb{C}^2}.
    \arrow["{\begin{psmallmatrix} 1 & 1 \end{psmallmatrix}}", from=1-2, to=1-1]
    \end{tikzcd}
\end{Example}

Given a rational $\cstar$-action on a vector space $U$ and $i \in \mathbb{Z}$ the subspace consisting of vectors $u \in U$ for which $\lambda \cdot u = \lambda^iu$ for all $\lambda \in \cstar$ is called the $i^{th}$ weight space. A $\cstar$-representation $U$ is the direct sum of its weight spaces.

\begin{Definition}
    Consider $M \in \crep$ of dimension vector $\textbf{m}$ and $t: \cstar \rightarrow \operatorname{G}_{\textbf{m}}.$ The morphism $t$ gives an action of $\cstar$ on $M_{\textbf{m}}$ and let \[j = \mathrm{max}\{|i|: \text{the } i^{th} \text{ weight space of } M_\textbf{m} \text{ is nonzero for the action of } t\}.\] We denote by $M(t) \in \Mod \mauro {j} {Q}$ the representation with underlying $Q$ representation given by $M$ and at vertex $x \in Q_0$ and loop $\ell^{x}_i$ we attach the idempotent projecting onto the $i^{th}$-weight space of $\cstar$ on $M_x.$
\end{Definition}

For $N \in \Mod \mauro {j} {Q}$ we denote by $n_i^x$ the idempotent $N_{\ell_i^x}.$ We obtain \[t_{N}: \cstar \rightarrow \operatorname{G}_{\textbf{n}}\] given by $\lambda \mapsto (t_{N,x}(\lambda))_{x \in Q_0}$ where $t_{N,x}(\lambda) = \displaystyle\sum_{i = -j}^{j} \lambda^i n_i^x.$ We have $t_{N,x}(\lambda) \in \operatorname{GL}_{n_x}(\mathbb{C})$ since \[t_{N,x}(\lambda^{-1}) = t_{N,x}(\lambda)^{-1}.\] 

This gives rise to a natural way to convert between these languages. In particular, give a pair $(M, t)$ we can consider $M(t)$ and for $M \in \Mod \mauro{j}{Q}$ we can consider $(R(N), t_{N}).$ This will allow us to study $\cstar$-actions on $\crep$ via $\Mod \mauro {j} {Q}.$ 

\begin{Lemma} \label{mauroquot}
    For any $j \in \mathbb{N}$ with $j > 0$ there is an isomorphism \[\mauro{j}{Q}/\langle \ell_{-j}^x, \ell_j^x, \forall x \in Q_0\rangle \cong \mauro{j-1}{Q}.\]
\end{Lemma}

\begin{proof}
     Recall that $\mauro{j}{Q} \cong \mathbb{C}Q^j/I$ where $I$ is generated by $\ell_i^x\ell_j^x$ for $i \neq j$ and $x \in Q_0$ and $e_x - \sum_{i=-j}^j \ell_i^x.$ We have not needed it yet but we will add a subscript $j$ to $I_j$ to show that it is the ideal for $\mauro{j}{Q}.$ Under the natural quotient $p: \mathbb{C}Q^j \twoheadrightarrow \mathbb{C}Q^{j-1}$ given by sending $\ell_j^x$ and $\ell_{-j}^x$ to zero we have $p(I_j) \subseteq I_{j-1}.$ By the universality of cokernels we get a surjection 
     \[\begin{tikzcd}[ampersand replacement=\&]
	{\mathbb{C}Q^j} \&\& {\mathbb{C}Q^{j-1}} \\
	\\
	{\mauro{j}{Q}} \&\& {\mauro{j-1}{Q}.}
	\arrow[two heads, from=1-1, to=1-3]
	\arrow[two heads, from=1-1, to=3-1]
	\arrow[two heads, from=1-3, to=3-3]
	\arrow[dashed, two heads, from=3-1, to=3-3]
    \end{tikzcd}\] This surjection induces the claimed isomorphism.
\end{proof}

\begin{corollary}
    The quotient $\mauro{j}{Q} \twoheadrightarrow \mauro{j-1}{Q}$ gives rise to the restriction functor \[i_j: \Mod \mauro{j-1}{Q} \rightarrow \Mod \mauro{j}{Q}\] for any $j > 0.$ Since it is a surjection then by Proposition \ref{restriction} this is an embedding.
\end{corollary}

\begin{Remark} \label{equiv remark}
    Note that $i_j (\Mod \tmauro {j-1} {Q} {T}) \subseteq \Mod \tmauro {j} {Q} {T}.$
\end{Remark}

We will now also define $i_0.$ 

\begin{Lemma}
    There is an isomorphism $\mauro{0}{Q} \cong \mathbb{C}Q.$
\end{Lemma}

\begin{proof}
    Recall we have an injection $\mathbb{C}Q \hookrightarrow \mauro{0}{Q}.$ Since $\mauro{0}{Q}$ is generated by the arrows in $Q$ and a loop $\ell_x^0$ at each vertex $x \in Q_0$ along with the relation $e_x - \ell_x^0.$ In particular, the subalgebra $\mathbb{C}Q$ generates all of $\mauro{0}{Q}$ and so this map is an isomorphism.
\end{proof}

\begin{corollary} \label{split embedding}
    The map $\mauro{0}{Q} \twoheadrightarrow \mathbb{C}{Q}$ gives a functor $i_{0}: \Mod \mathbb{C}Q \rightarrow \Mod \mauro {0} {Q}.$ Since the map is an isomorphism then this is an equivalence of categories.
\end{corollary}

Let $C_i,\, i \in \mathbb{Z}$ be a family of categories and $f_{ij}:C_i \rightarrow C_j$ functors for $i \leq j.$ For $X$ a category and $\phi_i: C_i \rightarrow X$ a functor for each $i \in \mathbb{Z},$ the pair $\langle X, \phi_i\rangle$ is called a target for the direct system $(C_i,f_{ij})$ if $\phi_i = \phi_j f_{ij}.$ The category $X$ is called the colimit of $\langle C_i,f_{ij} \rangle$ if for any other target $\langle Y, \psi_i\rangle$ of $(C_i,f_{ij})$ there is a unique functor up to natural equivalence $u:X\rightarrow Y$ such that $\psi_i= u\phi_i.$ Colimits do not have to exist.

\begin{Definition}
    Let $\widehat{Q}$ be the quiver with $\widehat{Q}_0 = Q_0$ and \[\widehat{Q}_1 = Q_1 \cup\{\ell_i^x: x \rightarrow x, \, \forall x \in Q_0,\, i \in \mathbb{Z}\}.\] Define the full subcategory \[\Mod \limmauro{Q} \subseteq \Mod \mathbb{C}\widehat{Q}\] by the objects in $\Mod \mathbb{C}\widehat{Q}$ such that \begin{itemize}
        \item at each vertex only finitely many maps on the loops are non-zero,
        \item the finite set of non-zero loops at each vertex give an orthogonal decomposition of the identity endomorphism at that vertex. 
    \end{itemize}
\end{Definition}

Given $\Mod \mauro{j}{Q}$ and $\Mod \mauro{k}{Q}$ with $k \geq j$ then we can construct \[i_{j,k} = i_k i_{k-1} \dotsb i_{j+1} \colon \Mod \mauro{j}{Q} \rightarrow \Mod \mauro{k}{Q}.\] This gives us a directed system $(\Mod \mauro{j}{Q}, i_{j,k}).$

\begin{Proposition}
    For the direct system $(\Mod \mauro{j}{Q}, i_{j,k})$ the colimit exists and is equivalent to \[\Mod \limmauro {Q}.\]
\end{Proposition}  

\begin{proof}
    By composing the surjections $\mathbb{C}\widehat{Q} \twoheadrightarrow \mathbb{C}Q^j$ and $\mathbb{C}Q^j \twoheadrightarrow \mauro{j}{Q}$ and using \ref{restriction} we get an embedding \[\phi_j: \Mod \mauro {j} {Q} \hookrightarrow \Mod \mathbb{C}\widehat{Q}.\] Since the surjective map $\mathbb{C}\widehat{Q} \twoheadrightarrow \mauro{j}{Q}$ factors via the quotient described in \ref{mauroquot} we see \[\phi_j = \phi_{j+1}i_{j+1}.\] Clearly $\mathrm{Im}(\phi_j) \subset \Mod \limmauro{Q}$ therefore $\langle \Mod \limmauro{Q}, \phi_j\rangle$ is a target for the direct system $(\Mod \mauro{j}{Q}, i_{j,k}).$ 

    Since the functors $\phi_j: \Mod \mauro{j}{Q} \rightarrow \Mod \limmauro{Q}$ are embeddings then there is a functor $\phi_i^{-1}: \mathrm{Im}(\phi_i) \rightarrow \Mod \mauro{i}{Q}.$ Given another target $\langle Y, \psi_i\rangle,$ of $(\Mod \mauro{j}{Q}, i_{j,k})$ every $Z \in \Mod \limmauro{Q}$ is in the image of $\phi_i$ for some $i \in \mathbb{N}$ so we define the functor $u: \limmauro{Q} \rightarrow Y$ by $u(Z) = \psi_i(\phi_i^{-1}(Z))$ for $Z\in \mathrm{Im}(\phi_i).$ This satisfies $\psi_i = u \phi_i.$ 
\end{proof}

\begin{corollary}
    Due to Remark \ref{equiv remark}, imposing the natural definition of $T$-homogeneity on $\Mod \limmauro{Q}$ will then give an analogous construction for $\Mod \tlimmauro {Q} {T}$ in $\mathrm{Idem}(Q).$ 
\end{corollary}

\begin{Definition} \label{global R}
    We define the functor \[R: \Mod \limmauro {Q} \rightarrow \Mod \mathbb{C}Q\] as the limit of the functors $R: \Mod \mauro{j}{Q} \rightarrow \Mod \mathbb{C}Q.$ That is to say \[R(\phi_i(M)) = R(M).\] This exists and is well defined since $R(i_j(M)) = R(M)$ for $M \in \Mod \mauro{j-1}{Q}.$
\end{Definition}

\begin{Definition}
    Given an algebra $A = \mathbb{C}Q/J$ then we define $\mauro{j}{A} = \mauro{j}{Q}/\tilde{J}$ where $\tilde{J}$ is the two-sided ideal in $\mauro{j}{Q}$ generated by $J.$
\end{Definition}

The quotient $\mathbb{C}Q \twoheadrightarrow A$ allows us to consider $\Mod A$ as a full subcategory of $\Mod \mathbb{C}Q.$ Likewise since $\mauro{j}{A}$ is defined as a quotient of $\mauro{j}{Q}$ then we can consider $\Mod \mauro{j}{A}$ as a full subcategory of $\Mod \mauro{j}{Q}.$ We are implicitly using these embeddings in the following result.

\begin{Lemma}
    The full subcategory $\Mod \mauro{j}{A} \subseteq \Mod \mauro{j}{Q}$ is characterised by \[\Mod \mauro{j}{A} = \{N \in \Mod \mauro{j}{Q}: R(N) \in \Mod A\}.\] 
\end{Lemma}

\begin{proof}
    First consider $N \in \Mod \mauro{j}{A}.$ By definition $\tilde{J}$ is contained in $\mathrm{Ann}_{\mauro{j}{Q}}(N) = \{r \in \mauro{j}{Q}: rn = 0, \forall n \in N\}.$ It then follows that $J$ is contained in $\mathrm{Ann}_{\mathbb{C}Q}(R(N)) = \{c \in \mathbb{C}Q: cm = 0, \forall m \in R(N)\}.$ The means that $R(N) \in \Mod A.$ Conversely, for $N \in \Mod \mauro{j}{Q}$ such that $R(N) \in \Mod A$ then since $\tilde{J}$ is generated by $J$ as a two-sided ideal in $\mauro{j}{Q}$ then $\tilde{J}$ must be contained in $\mathrm{Ann}_{\mauro{j}{Q}}(N)$ so $N \in \Mod \mauro{j}{A}.$ 
\end{proof}

\begin{Definition}\label{Alg mauro}
    We can use the above results to define the categories $\Mod \limmauro{A}$ and $\Mod \tlimmauro{A}{T}$ as before. These are full subcategories of $\Mod \limmauro{Q}$ and $\Mod \tlimmauro{Q}{T}$ characterised by $R(N) \in \Mod A$ as above. 
\end{Definition}

\section{Morphisms and actions}
\label{Morphisms}
Recall that we can think of a module $N$ in $\Mod \limmauro {Q}$ as a module $R(N)$ of the quiver $Q$ and finitely many non-zero maps $n^{x}_i$ at vertex $x \in Q_0$ which are orthogonal idempotents summing to the identity map on $N_x.$ We use this to study quiver Grassmannians and their Euler characteristics.

\begin{Lemma} \label{embed morph}
    For $\phi \in \Hom_{\limmauro {Q}}(M,N),$ we have $\phi_x = \displaystyle\sum_{i} n^{x}_i \phi_x m^{x}_i \hspace{0.2cm} \forall x \in Q_0.$
\end{Lemma}

\begin{proof}
    Since $\phi$ is a morphism of representations of quivers then \[n^{x}_i \phi_x = \phi_x m^{x}_i.\] Furthermore, we have \[\displaystyle\sum_{i} n^{x}_i = \mathrm{Id}_{N_x}.\] Apply this as follows \[\begin{aligned} \phi_x = & \left(\sum_i n_i^x\right) \phi_x \\
    = & \sum_i n_i^x \phi_x \\
    = & \sum_i n_i^x n_i^x \phi_x \\
    = & \sum_i n_i^x \phi_x m_i^x.        
    \end{aligned}\]
\end{proof}

\begin{corollary} \label{injective split}
    If $\phi \in \Hom_{\limmauro {Q}}(M,N)$ is injective (surjective) then $n^{x}_i \phi_x m^{x}_i$ is injective (surjective) when considered as a map $n^{x}_i \phi_x m^{x}_i: \mathrm{Im}(m_i^x) \rightarrow \mathrm{Im}(n^{x}_i).$
\end{corollary}

Recall that we think of the loops as recording a $\cstar$-action on the underlying module. The relation $n^{x}_i \phi_x = \phi_x m^{x}_i$ is equivalent to the morphisms $\phi_x$ being intertwiners for the $\cstar$-action at each vertex. Given a pair $M,N \in \Mod \limmauro{Q},$ there is a natural inclusion $\Hom_{Q}(R(M),R(N)) \subseteq \bigoplus_{x \in Q_0} \Hom_{\mathbb{C}}(M_x,N_x).$ Note that there is a $\cstar$-action on $\bigoplus_{x \in Q_0} \Hom_{\mathbb{C}}(M_x,N_x)$ given by $\lambda \cdot f = t_N(\lambda) f t_M(\lambda)^{-1}.$ We can describe the cases where $\Hom_{Q}(R(M),R(N)) \subseteq \bigoplus_{x \in Q_0} \Hom_{\mathbb{C}}(M_x,N_x)$ is stable under this $\cstar$-action. 

\begin{Lemma} \label{Hom fix}
    For $M,N \in \Mod \tlimmauro {Q} {T}$ the space $\Hom_Q(R(M),R(N))$ is stable under the $\cstar$-action and $\Hom_{\limmauro {Q}}(M,N) = \Hom_Q(R(M),R(N))^{\cstar}.$ 
\end{Lemma}

\begin{proof}
    Consider $\phi \in \Hom_Q(R(M),R(N)).$ The commutative square, 
    \[\begin{tikzcd}
	{M_x} && {M_y} \\
	\\
	{N_x} && {N_y}
	\arrow["{M_{a}}", from=1-1, to=1-3]
	\arrow["{\phi_x}"{description}, from=1-1, to=3-1]
	\arrow["{\phi_y}"{description}, from=1-3, to=3-3]
	\arrow["{N_{a}}", from=3-1, to=3-3]
    \end{tikzcd}\]
    gives rise to the square 
    \[\begin{tikzcd}
	{M_x} && {M_y} \\
	\\
	{N_x} && {N_y}
	\arrow["{M_{a}}", from=1-1, to=1-3]
	\arrow["{t_{N,x}(\lambda)\phi_x(t_{M,x}(\lambda))^{-1}}"', from=1-1, to=3-1]
	\arrow["{t_{N,y}(\lambda)\phi_y(t_{M,y}(\lambda))^{-1}}", from=1-3, to=3-3]
	\arrow["{N_{a}}", from=3-1, to=3-3]
    \end{tikzcd}\] under the action of $\lambda \in \cstar.$ Since $M\in \Mod \tlimmauro {Q} {T}$ then $t_{M,y}(\lambda) M_a t_{M,x}(\lambda)^{-1} = \lambda^{T_a}M_a$ and likewise for $N \in \Mod \tlimmauro {Q} {T}.$ This implies the new square commutes since \[\phi_y (t_{M,y}(\lambda))^{-1} M_{a} t_{M,x}(\lambda) = \lambda^{-T_{a}} \phi_y M_{a} = \lambda^{-T_{a}} N_{a}\phi_x =  (t_{N,y}(\lambda))^{-1}N_{a}t_{N,x}(\lambda) \phi_x.\] 

     Recall that $t_{N,x}(\lambda) = \sum_i \lambda^i n_i^x.$ Therefore a morphism is invariant if and only if it lies in the set $\{(\phi_x)_{x \in Q_0}: n_i^x \phi_x = \phi_x m_i^x, \, \forall i, x\}.$ It then follows that $\Hom_{\limmauro {Q}}(M,N)$ is the space of invariant maps for this action.    
\end{proof}

We now give the more general version of the above lemma. 

\begin{Lemma} \label{Hom dec}
    For $M,N \in \Mod \limmauro {Q}$ then \[\Hom_{\limmauro {Q}}(M,N) = \big(\bigoplus_{x \in Q_0} \Hom_{\mathbb{C}} (M_x,N_x)^{\cstar}\big) \cap \Hom_Q(R(M),R(N))\] where all three spaces are implicitly being considered as subspaces of \[\bigoplus_{x \in Q_0} \Hom_{\mathbb{C}} (M_x,N_x).\]
\end{Lemma}

\begin{proof}
    Since $\Mod \limmauro{Q}$ is the colimit of $\Mod \mauro{j}{Q},$ there is a $j \in \mathbb{N}$ such that $\Hom_{\limmauro Q}(M,N) = \Hom_{\mauro{j}{Q}}(M,N).$ By partitioning the arrows in $\mauro{j}{Q}$ into those of $Q$ and the loops $\ell_i^x,$ we see that \[\Hom_{\mauro{j}{Q}}(M,N) = \{(\phi_x)_{x \in Q_0}: n_i^x \phi_x = \phi_x m_i^x, \, \forall i, x\} \cap \Hom_Q(R(M),R(N)).\] The first set of maps is precisely the set of invariant maps.
\end{proof}

\begin{Definition}
    For a representation $M \in \crep$ and dimension vector $\textbf{d} \in \mathbb{N}^{Q_0},$ the quiver Grassmannian $\gr_\textbf{d}(M) \subseteq \prod_x \gr(d_x, M_x)$ is the subvariety consisting of the $(U_x)_{x \in Q_0} \in \prod_x \gr(d_x, M_x)$ such that $M_a(U_{s(a)}) \subseteq U_{t(a)}$ for all $a \in Q_1.$   
\end{Definition}

\begin{Proposition} \label{Grass fix}
    Fix $N \in \Mod \tlimmauro {Q} {T}$ and recall that we have an associated morphism $t_N: \cstar \rightarrow \operatorname{G}_{\textbf{n}}$ which defines a $\cstar$-action on $\bigoplus_x N_x.$ For any $\textbf{d} \in \mathbb{Z}^{Q_0}$ the morphism $t_N$ induces a $\cstar$-action on the quiver Grassmannian $\gr_\textbf{d}(R(N))$ and \[\gr_{\textbf{d}}(N) = \gr_{\textbf{d}}(R(N))^{\mathbb{C}^{\times}}.\] 
\end{Proposition}

\begin{proof}
    Define $\Hom^{0}(\textbf{d},R(N))$ to be the space of pairs $(U,f)$ where $U$ is a representation of $Q$ with dimension vector $\textbf{d}$ and $f$ an injective map $f\colon U \hookrightarrow R(N).$ Recall from \cite{CR} that we can identify $\gr_{\textbf{d}}(R(N))$ with the quotient \[\Hom^{0}(\textbf{d},R(N))/G_{\textbf{d}}.\]
    We will act via $\cstar$ by post-composing the maps in $\Hom^{0}(\textbf{d},R(N))$ with $t_N.$ Since embeddings in $\Mod \tlimmauro {Q} {T}$ are intertwiners for the $\cstar$-action, post-composing with the $\cstar$-action is equivalent to acting on the submodule. This acts via a subgroup of $G_{\textbf{d}}$ so is fixed in the quotient. This shows $\gr_{\textbf{d}}(N) \subset \gr_{\textbf{d}}(R(N))^{\cstar}.$ 

    Now fix $\phi = (\phi_x)_{x \in Q_0}$ a series of injective maps in $\Hom^{0}(\textbf{d},R(N))$ as a representative for a $\cstar$-fixed submodule $U \subset R(N).$ Since the image of $\phi_x$ is fixed by the action there is an induced $\cstar$-action on $U$ defined by $\phi(\lambda \cdot u) =\lambda \cdot \phi(u).$ This $\cstar$-action gives rise to a morphism $t: \cstar \rightarrow \mathrm{G}_\textbf{u}.$ Since $(\phi)_{x \in Q_0}$ is, by definition, an intertwiner for this action, $(\phi_x)_{x \in Q_0}$ defines $U(t)$ as a submodule of $N.$ This shows $\gr_{\textbf{d}}(R(N))^{\mathbb{C}^{\times}} \subset \gr_{\textbf{d}}(N)$ completing the proof.
\end{proof}

\begin{Remark}
    More generally, suppose that $\mathcal{F}_{\textbf{d}_1,\textbf{d}_2,\cdots,\textbf{d}_r}(M)$ is a variety of flags of submodules of fixed dimension vectors, an adaptation of the above argument says that for $N \in \Mod \tlimmauro {Q} {T}$ we have $\mathcal{F}_{\textbf{d}_1,\textbf{d}_2,\cdots,\textbf{d}_r}(N) = \mathcal{F}_{\textbf{d}_1,\textbf{d}_2,\cdots,\textbf{d}_r}(R(N))^{\mathbb{C}^\times}.$ 
\end{Remark}

\begin{Example}
    Consider $Q$ to be the quiver with one point, $x,$ and no arrows and fix $Y \in \Mod \mauro{2}{Q}.$ Note that in this case $\mauro{j}{Q} = \tmauro{j}{Q}{T}.$ Fix a dimension $d \in \mathbb{N}.$ We can represent $Y$ as 
    \[\begin{tikzcd}[ampersand replacement=\&]
	{Y_x}
	\arrow["{y_1}", from=1-1, to=1-1, loop, in=145, out=215, distance=10mm]
	\arrow["{y_2}", from=1-1, to=1-1, loop, in=325, out=35, distance=10mm]
    \end{tikzcd}\] where $y_1y_2 = y_2y_1 = 0$ and $\mathrm{Id}_{Y_x} = y_1 + y_2.$ A submodule is the data of $(U_x, u_1, u_2)$ and an injective map $\phi_x$ such that \[\begin{tikzcd}[ampersand replacement=\&]
	{Y_x} \\
	\\
	{U_x}
	\arrow["{y_1}", from=1-1, to=1-1, loop, in=145, out=215, distance=10mm]
	\arrow["{y_2}", from=1-1, to=1-1, loop, in=325, out=35, distance=10mm]
	\arrow["{\phi_x}"', hook, from=3-1, to=1-1]
	\arrow["{u_1}", from=3-1, to=3-1, loop, in=145, out=215, distance=10mm]
	\arrow["{u_2}", from=3-1, to=3-1, loop, in=325, out=35, distance=10mm]
    \end{tikzcd}\] with $y_1\phi_x = \phi_xu_1$ and $y_2\phi_x = \phi_xu_2.$ Note that $y_1\phi_xu_2 = y_2\phi_xu_1 = 0$ since $y_1$ and $y_2$ are orthogonal. It follows that $\phi_x = (y_1+y_2)\phi_x = y_1\phi_x + y_2\phi_x = y_1\phi_xu_1 + y_2\phi_xu_2.$ Denote $a = \dim \mathrm{Im}(u_1)$ and $b = \dim \mathrm{Im}(u_2).$ Since $\phi_x = (y_1+y_2)\phi_x = \phi_x(u_1+u_2)$ and $\phi_x$ is injective then $d = a+b.$ This submodule is in the image of the embedding \[\gr(a, \mathrm{Im}(y_1)) \times \gr(b, \mathrm{Im}(y_2)) \hookrightarrow \gr(a+b, \mathrm{Im}(y_1+y_2)) = \gr(d, Y_x).\] Considering this decomposition for all submodules of dimension $d$ we see 
    \[\gr_d(Y) = \displaystyle\bigsqcup_{a + b = d} \gr(a, \mathrm{Im}(y_1))\times \gr(b, \mathrm{Im}(y_2)) \subseteq \gr(d, Y_x) = \gr_d(R(Y)).\]
\end{Example}

\begin{Example} \label{split}
    Fix the quiver $Q,$ representations $L,M \in \crep$ and let $Y \in \Mod \limmauro {Q}$ be the module such that $R(Y) \cong L\oplus M$ with $Id_{L_x}$ on loop 0 and $Id_{M_x}$ on loop 1. This is in $\Mod \tlimmauro {Q} {T}$ for $T = (0,0,\cdots,0).$ Extending the previous example we see that \[\gr_{\textbf{d}}(Y) \cong  \displaystyle\bigsqcup_{\textbf{f} + \textbf{h} = \textbf{d}} \gr_{\textbf{f}}(L)\times \gr_{\textbf{h}}(M).\] This  subvariety \[\displaystyle\bigsqcup_{\textbf{f} + \textbf{h} = \textbf{d}} \gr_{\textbf{f}}(L)\times \gr_{\textbf{h}}(M) \subseteq \gr_\textbf{d}(L\oplus M)\] is considered when multiplying cluster characters. 
\end{Example} 

\begin{Definition}
    Given a variety, X, over $\mathbb{C}$, we denote by $\chi(X)$ the topological Euler characteristic with compact support of the variety of $\mathbb{C}$-points, i.e. $\chi(X) = \sum_i (-1)^i \dim \mathrm{H}_c^i(X(\mathbb{C}), \mathbb{C}).$
\end{Definition}

\begin{Proposition} \label{Euler chars}
    For $N$ in $\Mod \tlimmauro {Q} {T}$ and any dimension vector $\textbf{d} \in \mathbb{N}^{Q_0},$ then $\chi(\gr_{\textbf{d}}(N)) = \chi(\gr_{\textbf{d}}(R(N))).$
\end{Proposition}

\begin{proof}
    Since $\gr_{\textbf{d}}(N)$ is the subvariety of $\cstar$-fixed points of $\gr_\textbf{d}(R(N))$ they have the same Euler characteristic \cite{BB}. 
\end{proof} 

\begin{Proposition}
    Given $U \in \gr_{\textbf{d}}(N)$ for $N \in \Mod \limmauro Q$ it follows that $$\dim T_{U}(\gr_{\textbf{d}}(N)) \leq \displaystyle\sum_{x \in Q_0} \displaystyle\sum_{i \in \mathbb{Z}} \rank(u^{x}_i)(\rank(n^{x}_i) - \rank(u^{x}_i)).$$
\end{Proposition}

\begin{proof}
    The tangent space of a point in a quiver Grassmannian is known due to \cite{CR}. In particular \[T_{U}(\gr_{\textbf{d}}(N)) \cong \Hom_{\limmauro{Q}}(U,N/U).\] We denote $M = N/U.$ By Lemma \ref{embed morph} we see that \[\dim \Hom_{\limmauro{Q}}(U,M) \leq \dim \displaystyle\sum_{x \in Q_0} \displaystyle\sum_{i \in \mathbb{Z}} \dim  \Hom_{\mathbb{C}} (\mathrm{Im}(u_i^x),\mathrm{Im}(m_i^x)).\] Note that $\dim \mathrm{Im}(m_i^x) = \dim \mathrm{Im}(n_i^x) - \dim \mathrm{Im}(u_i^x)$ since $M = N/U.$ Therefore, \[\dim \Hom_{\mathbb{C}} (\mathrm{Im}(u_i^x),\mathrm{Im}(m_i^x)) = \rank u_i^x (\rank n_i^x - \rank u_i^x).\] This then gives the right hand side of the bound. 
\end{proof}

\begin{Definition}
    We say that $N \in \Mod \limmauro Q$ is primitive if all of the non-zero idempotents, $n_i^x,$ $i \in \mathbb{Z}$ $x \in Q_0$ are primitive (have rank one).
\end{Definition}

\begin{corollary}
    If $N \in \Mod \limmauro Q$ is primitive then $\gr_{\textbf{d}}(N)$ is zero dimensional for all $\textbf{d} \in \mathbb{N}^{Q_0}.$ 
\end{corollary}

\begin{Definition}
    Given $M \in \Mod \mathbb{C}Q$ we say that $M$ has a primitive equivariant lift if there is a $T \in \mathbb{Z}^{Q_1}$ and $N \in \Mod \tlimmauro {Q} {T}$ such that $N$ is primitive and $R(N) \cong M.$
\end{Definition}

\begin{corollary}
   If $M$ has a primitive equivariant lift then $\chi(\gr_{\textbf{d}}(M)) \geq 0,$ $\forall \textbf{d} \in \mathbb{N}^{Q_0}.$    
\end{corollary}

We may wonder if this approach could be used to compute Euler characteristics of non-trivial varieties with Euler characteristic zero. The following lemma tells us the answer is no.

\begin{Lemma} \label{no zeros}
    For $N \in \Mod \tlimmauro {Q} {T}$ then $\gr_{\textbf{d}}(R(N))$ is non-empty if and only if $\gr_{\textbf{d}}(N)$ is non-empty.
\end{Lemma}

\begin{proof}
    Since $\gr_\textbf{d}(N) \subseteq \gr_\textbf{d}(R(N))$ one implication is obvious. The converse is equivalent to the fact that any $\cstar$-action on a projective variety must have a fixed point. Fix a point $x \in \gr_{\textbf{d}}(R(N))$ and consider $\overline{\cstar \cdot x},$ the closure of its orbit. This is stable under the $\cstar$-action. If we suppose $x$ is not a fixed point, then its orbit is isomorphic to $\cstar.$ Since $\cstar$ is not closed, $\overline{(\cstar \cdot x)}\backslash (\cstar \cdot x)$ is a non-empty zero-dimensional $\cstar$-stable subvariety. Therefore the points in $\overline{(\cstar \cdot x)}\backslash (\cstar \cdot x)$ are fixed points and are therefore in $\gr_{\textbf{d}}(N).$ 
\end{proof}

\section{Equivariance}
In this section we will study the category $\Mod \tlimmauro{Q}{T}$ and show it is equivalent to a category of equivariant modules. Fix $T \in \mathbb{Z}^{Q_1}$ and for $a \in Q_1$ let $\lambda \cdot a = \lambda^{T_a}a.$ This extends to an action of $\cstar$ on the path algebra of Q. 

\begin{Definition}
    We denote by $\qccat$ the category whose objects are pairs $(M, t)$ where $M \in \Mod \mathbb{C}Q$ and $t: \cstar \rightarrow \vecend{M}$ is a morphism of algebraic groups such that $M_x$ is $\cstar$-stable for each vertex $x \in Q_0.$ A morphism in $\qccat$ is then a morphism of the $\mathbb{C}Q$-modules which also intertwines the $\cstar$-actions on these modules.
\end{Definition}

\begin{Lemma}
    Given a $\mathbb{C}Q$-module $M$ and $t: \cstar \rightarrow \vecend{M}$ a morphism of algebraic groups then the pair $(M, t)$ is an object in $\qccat$ if and only if $t: \cstar \rightarrow \vecend{M}$ factors via $\operatorname{G}_\textbf{m}.$
\end{Lemma}

\begin{proof}
    If $t$ factors via $\operatorname{G}_\textbf{m}$ then $M_x$ is $\cstar$-stable for each $x \in Q_0.$ Conversely, if $M_x$ is $\cstar$-stable for each $x \in Q_0$ then the morphism $t: \cstar \rightarrow \vecend{M}$ must factor via the subspace $\{g\in \vecend{M}: g(M_x) = M_x\} \subseteq \vecend{M}.$ But this subset is precisely $\operatorname{G}_\textbf{m}.$ 
\end{proof}

\begin{Definition}
    Given $M \in \Mod A$ and $G$ an algebraic group acting rationally on both $M$ and $A,$ we say $M$ is a G-equivariant $A$-module if the action map $$A \times M \rightarrow M$$ intertwines the G-actions. That is to say $(g \cdot a) \cdot (g \cdot m) = g \cdot (a \cdot m).$
\end{Definition}

\begin{Theorem} \label{equivariant}
    The category $\Mod \limmauro{Q}$ is equivalent to the category $\qccat.$
    The category of $\mathbb{C}^{\times}$-equivariant modules over $\mathbb{C}Q$ with the action induced by T is equivalent to $\Mod \tlimmauro {Q} {T}.$
\end{Theorem}

\begin{proof}
    Recall that by Lemma \ref{Hom dec} \[\Hom_{\limmauro {Q}}(M,N) = \big(\bigoplus_{x \in Q_0} \Hom (M_x,N_x)^{\mathbb{C}^{\times}}\big) \cap \Hom(R(M),R(N)).\] In particular, given a morphism $f \in \Hom_{\limmauro {Q}}(M,N)$ then $R(f): R(M) \rightarrow R(N)$ defines a morphism from $(R(M), t_M)$ to $(R(N), t_N).$ Likewise, given a morphism $g:(M, t_1) \rightarrow (N, t_2)$ in $\qccat$ then under Lemma \ref{Hom dec} we can identify this with a morphism $G(g): M(t_1) \rightarrow N(t_2)$ such that $R(G(g)) = g.$ Furthermore we see that $G(R(f)) = f.$ This allows us to define the following two functors.
    
    Let $F\colon\Mod \limmauro{Q} \rightarrow \qccat$ be the functor defined by $F(M) = (R(M), t_M)$ on objects and $F(f) = R(f)$ for morphisms. Furthermore, define $G\colon \qccat \rightarrow \Mod \limmauro{Q}$ by $G(N,t) = N(t)$ on objects and $G(f)$ is defined as above for morphisms. These functors give us the equivalence of categories since $GF = \mathrm{Id}$ and $FG = \mathrm{Id}.$ 
    
    It remains to show that equivariance with respect to the action induced by $T$ on $\mathbb{C}Q$ is equivalent to being in $\Mod \tlimmauro {Q} {T}.$ Firstly, we see that if $M$ is an equivariant module then this implies $e_x(\lambda \cdot m) = \lambda \cdot (e_x\cdot m).$ Recall that a morphism in $\qccat$ is a morphism of the underlying $\mathbb{C}Q$-modules which interwines the $\cstar$ actions. A morphism in the equivariant category is a morphism of the underlying $\mathbb{C}Q$-modules which interwines the $\cstar$-actions. This allows us to consider morphism in the equivariant category as morphisms in $\qccat$ and vice versa. In particular, there is an embedding of the category of equivariant modules into $\qccat$ given by sending $M$ with its action $t_M$ to the pair $(M, t_M)$ and morphisms in the category of equivariant modules are sent to the corresponding $\cstar$-intertwiner in $\qccat.$

   Consider $a \in M_x$ and $b \in Q_1$ with $s(b) = x,$ $t(b) = y.$ We want to show that for any $a, b$ as above then equivariance is equivalent to $\lambda^{T_b}M_b = t_{M,y}(\lambda) M_b t_{M,x}(\lambda)^{-1}.$ First, observe 
   \begin{align*}(\lambda \cdot b) \cdot (\lambda \cdot a) &= \lambda^{T_b} b \cdot (t_{M,x}(\lambda) a) \\ &= \lambda^{T_b}M_b(t_{M,x}(\lambda)a).
   \end{align*}

   On the other hand 
   \begin{align*}
       \lambda \cdot (b \cdot a) &= \lambda \cdot M_b(a) \\
       &= t_{M,y}(\lambda)M_b(a).
   \end{align*}

   Equivariance is equivalent to equality \[\lambda^{T_b}M_b(t_{M,x}(\lambda)a) = t_{M,y}(\lambda)M_b(a)\] for any $a \in M_x$ and $b \in Q_1$ with $t(b) = y,$ $s(b) = x.$ By first applying $t_{M,x}(\lambda)^{-1}$ to $M_x$ we this is the condition that \[\lambda^{T_b}M_b(a) = t_{M,y}(\lambda)M_b(t_{M,x}(\lambda)^{-1}a)\] for any $a \in M_x.$ In particular, it is equivalent to $\lambda^{T_b}M_b = t_{M,y}(\lambda)M_bt_{M,x}(\lambda)^{-1}$ which is precisely the condition for $T$ homogeneity.
\end{proof}

\section{Idempotents, gradings, and coefficient quivers}
Given $N \in \Mod \tlimmauro {Q} {T}$ we have shown $\chi(\gr_{\textbf{d}}(N)) = \chi(\gr_{\textbf{d}}(R(N))).$ In particular, we are interested in the case where $N$ is primitive equivariant, i.e. $\gr_\textbf{d}(N)$ is finite, since then $|\gr_{\textbf{d}}(N)| = \chi(\gr_{\textbf{d}}(R(N))).$

\begin{Definition}
    Given a representation $M \in \crep$ and $B = \cup_{x \in Q_0} B_x$ where $B_x$ is a basis at vertex $x,$ we call a map $\partial : B \rightarrow \mathbb{Z}$ a grading. 
\end{Definition}

Note that these gradings are not required to respect any gradings on the algebra $\mathbb{C}Q.$ These are not equivalent to a grading on $\mathbb{C}Q$ in the standard sense.

A grading, $\partial,$ for a basis of $M \in \crep$ defines a collection of idempotents for $M$ as follows. For $B_x= \{b_{1}, \cdots, b_{r}\}$ the basis at vertex $x \in Q_0$ let the idempotent associated to loop $j$ be the idempotent projecting onto the subspace given by $\spn_{\mathbb{C}}\{b_k: \overline{\partial}(b_k) = j\}$ and such that all of the idempotents at any given vertex are orthogonal. These two conditions together uniquely determine the choices of idempotents at each vertex. Likewise, given a representation $M \in \Mod \limmauro Q,$ we can recover a grading by picking a basis of each $\mathrm{Im}(m_i^x)$ and giving them a grading such that $\partial(b_k) = j$ if $b_k \in \mathrm{Im}(m_j^x)$ for some $x \in Q_0.$ Notice that since $\sum_i m_i^x = \mathrm{Id}_{M_x}$ for any $x\in Q_0$ then this process genuinely defines a basis of $M_x$ at each vertex.

\begin{Lemma} \label{cerulli-irelli}
    A representation $M \in \crep$ has a primitive equivariant lift if and only if there is a basis $B = \cup_x B_x$ and a grading $\partial$ such that the map $\partial_x\colon B_x \rightarrow \mathbb{Z}$ is injective and for every arrow $a\colon i\rightarrow j$ and every $b_k, b_t \in B_i,$ the grading must satisfy the following two conditions:
    \begin{itemize}
        \item $M_a (b_k)$ is a multiple of a basis vector in $B_j$ ,
        \item given $b_r,b_t \in B_i$ and $a\colon i\rightarrow j$ such that $M_{a}(b_r) \neq 0$ and $M_{a}(b_t) \neq 0,$ $$\partial(M_{a}(b_r)) - \partial(M_{a}(b_t)) = \partial(b_r) - \partial(b_t).$$
    \end{itemize}
\end{Lemma}

\begin{proof}
    An object $N \in \Mod \limmauro {Q}$ with $R(N) \cong M$ is primitive if and only if $\partial_x$ is injective for $\partial$ the grading on $M$ induced by $N$ as above. Given $b_k \in B_i$ with $\partial_i(b_k) = k$ and $a\colon i \rightarrow j,$ the condition of $T$-homogeneity ensures that $M_a(b_k)$ is a scalar multiple of another basis element and $\partial(M_a(b_k)) - \partial(b_k) = T_a.$ Since this does not depend on $k$ we have that \[\partial(M_a(b_k)) - \partial(b_k) = \partial(M_a(b_l)) - \partial(b_l)\] and rearranging this gives the second condition.

    Conversely, given a grading satisfying the above condition, the injectivity of $\partial_x$ ensures the lift is primitive as already observed. Likewise, the second condition ensures that $\partial(M_a(b_k)) - \partial(b_k)$ is a constant for all $k.$ Denote this constant by $T_a$ and we use this to construct $T \in \mathbb{Z}^{Q_1}$ and $N \in \Mod \tlimmauro {Q} {T}$ primitive with $R(N) \cong M.$
\end{proof}

\begin{Remark}
    For a representation $M \in \crep$ and grading $\partial$ satisfying the above conditions, the associated module $N \in \Mod \limmauro{Q}$ satisfies $|\gr_\textbf{d}(N)| = \chi(\gr_\textbf{d}(M)).$ This is equivalent to the main theorem of \cite{CI}. 
\end{Remark}

The approach used by Cerulli-Irelli was generalised by ideas of Haupt \cite{HAU} as follows. Consider the case where $N \in \Mod \tlimmauro {Q} {T}$ and $L \in \Mod \limmauro {Q}$ with $R(L) = R(N)$ and the $\cstar$-action given by $t_{L}$ does not act on $\gr_{\textbf{d}}(R(N))$ but does act on the subvariety $\gr_{\textbf{d}}(N).$ This implies $\gr_{\textbf{d}}(L) \cap \gr_{\textbf{d}}(N) \subseteq \gr_{\textbf{d}}(R(N))$ is a smaller subvariety with the same Euler characteristic. We will now convert this idea into our representation-theoretic framework.

Much of the combinatorics of Euler characteristics has now been phrased in terms of coefficient quivers (for example \cite{JS}). These were introduced by Ringel \cite{R2} to capture properties of indecomposable and exceptional modules. 

\begin{Definition}[\cite{R2}]
    Consider $M \in \Mod \mathbb{C}Q$ and a fixed basis $B_x$ at each vertex $x \in Q_0$. The set of vertices of the coefficient quiver $\Gamma(B,M)$ is given by $\cup_{x \in Q_0} B_x.$ Given $b \in B_x$ and $b' \in  \mathcal{B}_y$ we add an arrow between the corresponding vertices if there is an arrow $\alpha\colon x \rightarrow y$ such that when we express $M_{\alpha}(b)$ in terms of the basis $B_y$ then its $b'$ coefficient is non-zero.     
\end{Definition}

\begin{Example}
    Consider the module $M,$ 
    \[
    \begin{tikzcd}[ampersand replacement = \&]
	\&\& {\mathbb{C}} \\
	{\mathbb{C}} \& {\mathbb{C}^{2}} \\
	\&\& {\mathbb{C}}
	\arrow["\begin{array}{c} \begin{psmallmatrix} 1 \\ 1 \end{psmallmatrix} \end{array}", from=2-1, to=2-2]
	\arrow["{\begin{psmallmatrix} 1 & 0 \end{psmallmatrix}}", from=2-2, to=1-3]
	\arrow["{\begin{psmallmatrix} 0 & 1 \end{psmallmatrix}}", from=2-2, to=3-3]
    \end{tikzcd}
    \]
    and let $B$ denote the standard basis at each vertex with respect to which the maps are written. Let $B'$ be the same basis except at the central vertex where we use the basis $b_1 + b_2$ and $b_1 - b_2.$ We have the two following coefficient quivers.

    $\Gamma(B,M):$
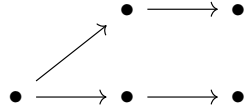
    \[
    \begin{tikzcd}
	& \bullet & \bullet \\
	\bullet & \bullet & \bullet
	\arrow[from=1-2, to=1-3]
	\arrow[from=2-1, to=1-2]
	\arrow[from=2-1, to=2-2]
	\arrow[from=2-2, to=2-3]
    \end{tikzcd}
    \]

    $\Gamma(B',M):$
    
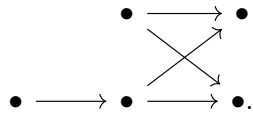
\[
    \begin{tikzcd}
	& \bullet & \bullet \\
	\bullet & \bullet & \bullet.
	\arrow[from=1-2, to=1-3]
	\arrow[from=1-2, to=2-3]
	\arrow[from=2-1, to=2-2]
	\arrow[from=2-2, to=1-3]
	\arrow[from=2-2, to=2-3]
    \end{tikzcd}
    \]
\end{Example}

\begin{Proposition}[Property 1, \cite{R2}]
    M is indecomposable if and only if $\Gamma(B,M)$ is connected for every choice of basis $B.$ 
\end{Proposition}

We say a module $M \in \crep$ is rigid if $\ext_{Q}^{1}(M,M) = 0.$
\begin{Theorem}[Theorem 1, \cite{R2}]
    If $M$ is indecomposable and rigid there is a basis $B$ such that $\Gamma(B,M)$ is a tree quiver. 
\end{Theorem}

\begin{Remark}
    The converse to this is not true. Consider the indecomposable module
    \[
    \begin{tikzcd}
	{\mathbb{C}} & {\mathbb{C}.}
	\arrow["0"', shift right=2, from=1-1, to=1-2]
	\arrow["1", shift left=2, from=1-1, to=1-2]
    \end{tikzcd}
    \]
    This module is not rigid but the only possible coefficient quiver is a tree.    
\end{Remark}

\begin{Definition}
    Given $N \in \Mod \limmauro {Q}$ we construct a quiver denoted by $\Gamma(N).$ The vertices are given by the set $\{n_i^x: n_i^x \neq 0\}.$ For each arrow $a: x \rightarrow y$ of $Q$ we add an arrow between the vertices corresponding to idempotents $n_{i}^{x}$ and $n_{j}^{y}$ if $n_{j}^{y} N_{a} n_{i}^{x} \neq 0.$
\end{Definition}

\begin{Example}
    \begin{itemize}
        \item For $L \in \Mod \mathbb{C}Q,$ we have that $\Gamma(i_{0}(L))$ is obtained from $Q$ by deleting an arrow $a$ if $L_{a} = 0$ and deleting a vertex $x$ if $L_x = 0.$
        \item Let $\widehat{M} \in \Mod \limmauro {Q}$ be the module induced from $M \in \Mod \mathbb{C}Q$ by a grading $\partial: \bigcup_x B_x \rightarrow \mathbb{Z}$ as before. It follows that $\Gamma(\widehat{M}) = \Gamma(B,M).$
    \end{itemize}
\end{Example}

\begin{Definition} \label{little gamma}
    We construct a representation $\gamma(M) \in \Rep(\Gamma(M),\mathbb{C})$ for $M \in \Mod \limmauro{Q}.$ At a vertex associated to the idempotent $m_i^x$ we assign the vector space $\mathrm{Im}(m_i^x)$ and for an edge in $\Gamma(M)$ defined by $m_j^yM_a m_i^x \neq 0$ we assign the map from $\mathrm{Im}(m_i^x)$ to $\mathrm{Im}(m_j^y)$ induced by $m_j^y M_a m_i^x.$    
\end{Definition}

\begin{Example}
    For $Q$ the quiver \[\begin{tikzcd}
	\bullet & \bullet & \bullet
	\arrow[from=1-2, to=1-1] \arrow[from=1-2, to=1-3]
    \end{tikzcd}\] let $M \in \Mod \limmauro{Q}$ be given by 
    
    \[\begin{tikzcd}[ampersand replacement = \&]    
    {\mathbb{C}^2} \& {\mathbb{C}^2} \& {\mathbb{C}^2}    
    \arrow[loop above, from= 1-1, to=1-1, "{\begin{psmallmatrix} 1 & 0 \\ 0 & 0\end{psmallmatrix}}"]
    \arrow[loop below, from= 1-1, to=1-1, "{\begin{psmallmatrix} 0 & 0 \\ 0 & 1\end{psmallmatrix}}"]    
    \arrow[from=1-2, to=1-1, "{\begin{psmallmatrix} 1 & 0 \\ 1 & 0\end{psmallmatrix}}"]      
    \arrow[loop above, from=1-2, to=1-2, "{\begin{psmallmatrix} 1 & 0 \\ 0 & 1\end{psmallmatrix}}"] 
    \arrow[loop below, from=1-2, to=1-2, "{\begin{psmallmatrix} 0 & 0 \\ 0 & 0\end{psmallmatrix}}"]
    \arrow[from=1-2,to=1-3,"{\begin{psmallmatrix} 1 & 0 \\ 0 & 1\end{psmallmatrix}}"]
    \arrow[loop above, from= 1-3, to=1-3, "{\begin{psmallmatrix} 1 & 0 \\ 0 & 0\end{psmallmatrix}}"]
    \arrow[loop below, from= 1-3, to=1-3, "{\begin{psmallmatrix} 0 & 0 \\ 0 & 1\end{psmallmatrix}}"] 
    \end{tikzcd}\]

    We only write the non-zero loops and do not worry about which loop they have been assigned to as it does not affect $\Gamma$ or $\gamma.$ We have $\Gamma(M)$ given by \[\begin{tikzcd}[ampersand replacement=\&]
	\bullet \&\&\&\& \bullet \\
	\&\& \bullet \\
	\bullet \&\&\&\& \bullet
	\arrow[from=2-3, to=1-1]
	\arrow[from=2-3, to=1-5]
	\arrow[from=2-3, to=3-1]
	\arrow[from=2-3, to=3-5]
    \end{tikzcd}\] and $\gamma(M)$ as 
    \[\begin{tikzcd}[ampersand replacement=\&]
	{\mathbb{C}} \&\&\&\& {\mathbb{C}} \\
	\&\& {\mathbb{C}^2} \\
	{\mathbb{C}} \&\&\&\& {\mathbb{C}}
	\arrow["{\begin{psmallmatrix} 1 & 0 \end{psmallmatrix}}"', from=2-3, to=1-1]
	\arrow["{\begin{psmallmatrix} 1 & 0 \end{psmallmatrix}}", from=2-3, to=1-5]
	\arrow["{\begin{psmallmatrix} 1 & 0 \end{psmallmatrix}}", from=2-3, to=3-1]
	\arrow["{\begin{psmallmatrix} 0 & 1 \end{psmallmatrix}}"', from=2-3, to=3-5]
\end{tikzcd}\]
\end{Example}

\begin{Definition}
    Given $N \in \Mod \limmauro Q$ and $\gamma(N) \in \Rep(\Gamma(N),\mathbb{C})$ we denote by \[\delta_{N} : \Gamma(N)_0 \rightarrow Q_0\] the map which projects back to the original vertices. That is to say, for the vertex, $y,$ of $\Gamma(N)$ defined by a non-zero idempotent $n_i^x$ then $\delta_N(y) = x.$ This induces a map of dimension vectors \[\underline{\delta}_{N} : \mathbb{N}^{\Gamma(N)_0} \rightarrow \mathbb{N}^{Q_0}\] by $\underline{\delta}_N(\textbf{d})_x = \displaystyle\sum_{\substack{y \in \Gamma(N)_0 \\ \delta_N(y) = x}} \textbf{d}_y.$
\end{Definition}

\begin{Lemma} \label{splitting up}
    Given $N \in \Mod \limmauro Q$ we have \[\gr_{\textbf{d}}(N) \cong \displaystyle\bigsqcup_{\underline{\delta}_{N} (\textbf{e}) = \textbf{d}} \gr_{\textbf{e}}(\gamma(N)).\]
\end{Lemma}

\begin{proof}
    Recall that a non-zero idempotent $n_i^x$ corresponds to a vertex $y \in \Gamma(N)_0$ and $\gamma(N)_y = \mathrm{Im}(n_i^x).$ Define the vector space isomorphism $p_x: N_x \rightarrow \displaystyle\bigoplus_{\delta_N(y) = x}\gamma(N)_y$ by $p_x(u) = (n_i^x(u))_{i \in J_x}$ where $J_x \subset \mathbb{Z}$ denotes the indices such that $i \in J_x$ if $n_i^x \neq 0.$ Combining these maps at each vertex we get an isomorphism \[p: \displaystyle\prod_{x \in Q_0} \gr(d_x, N_x) \rightarrow \displaystyle\prod_{x \in Q_0} \gr(d_x, \displaystyle\bigoplus_{\delta_N(y) = x} \gamma(N)_y).\] 

    Combining Lemma \ref{Hom dec} and the proof of Proposition \ref{Grass fix} \[\gr_{\textbf{d}}(N) = \gr_\textbf{d}(R(N)) \cap \displaystyle\prod_{x \in Q_0} \gr(\textbf{d}_x, N_x)^{\cstar}.\] By Lemma \ref{embed morph} and Lemma \ref{Hom dec} \[p(\displaystyle\prod_{x \in Q_0} \gr(\textbf{d}_x, N_x)^{\cstar}) = \displaystyle\bigsqcup_{\underline{\delta}_{N} (\textbf{e}) = \textbf{d}} \displaystyle\prod_{y \in \Gamma(N)_0} \gr(\textbf{e}_y, \gamma(N)_y).\]

    The point $(U_x)_{x \in Q_0} \in \displaystyle\prod_{x \in Q_0} \gr(d_x, N_x)$ is in $\gr_\textbf{d}(R(N))$ if and only if $N_a(U_{s(a)}) \subseteq U_{t(a)}$ for each $a \in Q_1.$ 
    Note $N_a = \displaystyle\sum_{i,j \in \mathbb{Z}} n_j^{t(a)} N_a n_i^{s(a)}.$ Therefore points in $\gr_\textbf{d}(R(N))$ are characterised by \[\displaystyle\sum_{i,j \in \mathbb{Z}} n_j^{t(a)} N_a n_i^{s(a)}(U_{s(a)}) \subseteq U_{t(a)}.\] This in turn is equivalent to \[\displaystyle\sum_{i,j \in \mathbb{Z}} n_j^{t(a)} N_a n_i^{s(a)}(n_i^{s(a)}U_{s(a)}) \subseteq \sum_{j \in \mathbb{Z}} n_j^{t(a)} U_{t(a)}.\] But since these idempotents are orthogonal this is equivalent to checking \[n_j^{t(a)} N_a n_i^{s(a)}(n_i^{s(a)}U_{s(a)}) \subseteq n_j^{t(a)} U_{t(a)}\] for all $i, j \in \mathbb{Z}$ such that $n_j^{t(a)}$ and $n_i^{s(a)}$ are non-zero. We see $p(\gr_\textbf{d}(N))$ is the subvariety of points in  \[\displaystyle\bigsqcup_{\underline{\delta}_{N} (\textbf{e}) = \textbf{d}} \displaystyle\prod_{y \in \Gamma(N)_0} \gr(\textbf{e}_y, \gamma(N)_y)\] satisfying $n_j^{t(a)}N_a n_i^{s(a)} (U_y) \subseteq U_z$ where $y, z \in \Gamma(N)_0$ correspond to the non-zero idempotents $n_i^{s(a)}$ and $n_j^{t(a)}$ respectively. This is precisely the condition for being a submodule of $\gamma(N)$ so we have the required isomorphism.
\end{proof}

\section{Equivariant sequences and Euler characteristics}
We present a theorem which is the combinatorial/representation theoretic version of previous work of Cerulli-Irelli \cite{CI} and Haupt \cite{HAU}. 

We have constructed a functor $R$ in Definition \ref{global R} and a map $\gamma$ in Definition \ref{little gamma} such that given $M \in \Mod \tlimmauro{Q}{T}$ we have

\begin{tikzcd}[ampersand replacement=\&]
	\&\& {M \in \Mod \tlimmauro{Q}{T}} \&\& \\
	\\
	{\gamma(M) \in \Mod \Gamma(M)} \&\&\&\& {R(M) \in \Mod \mathbb{C}Q.}
	\arrow["\gamma"', from=1-3, to=3-1]
	\arrow["R", from=1-3, to=3-5]
\end{tikzcd}

\vspace{0.5cm}

Furthermore, by Lemma \ref{splitting up} and Proposition \ref{Grass fix}, we have \[\chi(\gr_\textbf{d}(M)) = \gr_\textbf{d}(R(M)) = \chi(\displaystyle\bigsqcup_{\underline{\delta}_{M} (\textbf{e}) = \textbf{d}} \gr_{\textbf{e}}(\gamma(M))) = \displaystyle\sum_{\underline{\delta}_{M} (\textbf{e}) =  \textbf{d}} \chi(\gr_{\textbf{e}}(\gamma(M)))).\] Typically $\gr_{\textbf{e}}(\gamma(M)))$ is a lower dimensional variety and we can try to apply this approach again to form a new diagram

\begin{tikzcd}[ampersand replacement=\&]
	\&\& {M^1 \in \Mod \tlimmauro{\Gamma(M)}{T}} \&\& \\
	\\
	{\gamma(M^1) \in \Mod \Gamma(M^1)} \&\&\&\& {R(M^1) = \gamma(M) \in \Mod \Gamma(M).}
	\arrow["\gamma"', from=1-3, to=3-1]
	\arrow["R", from=1-3, to=3-5]
\end{tikzcd}

\vspace{0.5cm}

In general, this will reduce the dimensions of the quiver Grassmannians again. This approach is summarised in the proposition below.

\begin{Remark}
  In this section we will need to index sequences of modules over certain $\mauro{j}{Q}.$ Due to the abundance of indices in this paper we index them by the slightly unconventional $M^0, \dots, M^i.$ We apologise to the reader for this notation but see no better choice.  
\end{Remark}

\begin{Proposition}\label{haupt}
    Consider $M \in \crep$ and $r \in \mathbb{N}.$ Let $M^0 = i_0(M) \in \Mod \limmauro Q$ and choose $M^1 \in \Mod \tlimmauro {\Gamma(M^0)} {T_1}$ for some $T_1 \in \mathbb{Z}^{\Gamma(M^0)_1}$ with $R(M^1) \cong \gamma(M^0).$ Repeat this process for $1 \leq j \leq r$ by finding $M^j \in \Mod \tlimmauro {\Gamma(M^{j-1})} {T_j}$ for some $T_j \in \mathbb{Z}^{\Gamma(M^{j-1})_1}$ with $R(M^j) \cong \gamma(M^{j-1}).$ Consider $\delta^{r} \defeq \delta_r \cdots \delta_1,$ and the map $\underline{\delta}^{r}: \mathbb{N}^{\Gamma(M^{r-1})_0} \rightarrow \mathbb{N}^{Q_0}.$
    $$\chi(\gr_{\textbf{d}}(M)) = \displaystyle\sum_{\underline{\delta}^{r}(\textbf{e}) = \textbf{d}} \chi(\gr_{\textbf{e}}(M^r))$$ for all $\textbf{d} \in \mathbb{N}^{Q_0}.$
\end{Proposition} 

\begin{proof}
    This follows from repeatedly applying Lemma \ref{splitting up} and Proposition \ref{Grass fix}.
\end{proof}

We refer to a sequence $(M^0,\dotsc,M^r)$ as above as a sequence of equivariant lifts.

\begin{Definition} \label{limit}
    Given $M \in \crep$ and a sequence of equivariant lifts $(M^0,\dots,M^r)$ as above, the idempotents of $M^r$ induce a module $X \in \mathrm{Idem}(Q)$ as was considered in the previous section. This module is unique up to permuting the loops of $\mathrm{Idem}(Q).$ We call any such $X \in \mathrm{Idem}(Q)$ a \textit{limit} of $(M^0,\cdots,M^r).$
\end{Definition}

See Example \ref{lim ex} for an example of this process. 

\begin{Theorem} \label{limit euler}
    Let $X$ be a limit of a sequence of equivariant lifts $(M^0,\dots,M^r)$ with $M = R(M^0).$
    \[\gr_{\textbf{d}}(X) \cong \displaystyle\bigsqcup_{\underline{\delta}_M(\textbf{e}) = \textbf{d}} \gr_{\textbf{e}}(M^r).\] In particular \[\chi(\gr_{\textbf{d}}(X)) = \chi(\gr_{\textbf{d}}(M)).\]
\end{Theorem}

\begin{proof}
    This follows from combining \ref{splitting up} and \ref{haupt}.
\end{proof}

We will now show that sequences of equivariant lifts are the representation theoretic version of Haupt's nice sequences \cite{HAU}.
\begin{Definition}[Definition 4.2, \cite{HAU}] \label{haupt def}
    Given a sequence $(\partial_0,\dots,\partial_{i-1})$ of gradings, we say $\partial_i$ is a nice $(\partial_0,\dots,\partial_{i-1})$-grading on $M \in \crep$ if there exists a function \[\Delta : \mathbb{Z}^i \times \mathbb{Z}^i \times Q_1 \rightarrow \mathbb{Z}\] such that for any arrow $a \in Q_1$ and pair $e_t, e_j \in B$ such that $M_a(e_t) = \sum_k m_k e_k$ with $m_j \neq 0,$ \[\Delta((\partial_l(e_t))_{0 \leq l \leq i-1}, (\partial_l(e_j))_{0 \leq l \leq i-1}, a) = \partial_i(e_t) - \partial_i(e_j).\]
\end{Definition}

In a sequence of equivariant lifts $(M^0,\dots,M^r),$ with $M = R(M^0).$ We can choose a basis of $M^r$ with each basis vector in the image of an idempotent of $M^r.$ This basis will induce a basis of $M.$ This basis allows us to associate to $M^i$ a grading $\partial_i$ for M.

\begin{Proposition} \label{limit to haupt}
    Let $(M^0, \dots, M^r)$ be a sequence of modules with $R(M^i) = \gamma(M^{i-1})$ for each $i$ and $M^0 = i_0(M).$ The sequence $(M^0,\dots,M^r)$ is a sequence of equivariant lifts if and only if $\partial_i$ is a nice $(\partial_0,\dots,\partial_{i-1})$-grading for all $0 \leq i \leq r.$
\end{Proposition}

\begin{proof}
    First suppose $(M^0,\dots,M^r)$ is a sequence of equivariant lifts for $M = R(M^0).$ We can define $\Delta$ by setting \[\Delta((\partial_l(e_t))_{0 \leq l \leq i-1}, (\partial_l(e_j))_{0 \leq l \leq i-1}, a) = \partial_i(e_t) - \partial_i(e_j)\] for any $a,e_i,e_j$ as in Definition \ref{haupt def}. It just remains to check that if we have another triple $a,e_k,e_v$ satisfying the conditions of Definition \ref{haupt def} such that \[(\partial_l(e_t))_{0 \leq l \leq i-1}, (\partial_l(e_j))_{0 \leq l \leq i-1} = (\partial_l(e_k))_{0 \leq l \leq i-1}, (\partial_l(e_v))_{0 \leq l \leq i-1}\] then $\partial_i(e_t) - \partial_i(e_j) =  \partial_i(e_k) - \partial_i(e_v).$ If \[(\partial_l(e_t))_{0 \leq l \leq i-1}, (\partial_l(e_j))_{0 \leq l \leq i-1} = (\partial_l(e_k))_{0 \leq l \leq i-1}, (\partial_l(e_v))_{0 \leq l \leq i-1}\] then $e_t$ and $e_k$ are both in the image of a unique idempotent in $M^{i-1}$ and $e_j$ and $e_v$ are also both in the image of a unique idempotent in $M^{i-1}.$ It follows from the equivariance of $M^i$ over $\gamma(M^{i-1})$ that $\partial_i(e_t) - \partial_i(e_j) =  \partial_i(e_k) - \partial_i(e_v).$

    Conversely $(\partial_0,\cdots,\partial_r)$ such that $\partial_i$ is a nice $(\partial_0,\cdots,\partial_{i-1})$-grading for all $0 \leq i \leq r.$ Suppose $e_t$ and $e_k$ are both in the image of a unique idempotent in $M^{i-1}$ and $e_j$ and $e_v$ also both in the image of a unique idempotent in $M^{i-1}.$ Since the nice $(\partial_0,\cdots,\partial_{i-1})$-grading condition implies that $\partial_i(e_t) - \partial_i(e_j) =  \partial_i(e_k) - \partial_i(e_v)$ then for $T_a = \partial_i(e_t) - \partial_i(e_j)$ and extending this to $T \in \mathbb{Z}^{Q_1}$ this means $M^i$ is an equivariant lift of $\gamma(M^{i-1}).$
\end{proof}

This tells us that at the level of the corresponding gradings this is then Haupt's theory of \textit{nice sequences}. When our final module $M^r$ is a primitive equivariant lift of $\gamma(M^{r-1})$ this implies the corresponding nice sequence \textit{distinguishes bases} \cite{HAU}.

\section{Morphisms of quivers}
We have seen that representation theory in $\Mod \limmauro Q$ can capture much of the combinatorics that has previously been considered in terms of coefficient quivers and gradings. We will fully describe the relation between the two. 

A morphism of quivers $F\colon S \rightarrow Q$ is the data of $F_0\colon S_0 \rightarrow Q_0$ and $F_1\colon S_1 \rightarrow Q_1$ such that for $a \in S_1,$ we have $F_1(a)\colon F_0(s(a)) \rightarrow F_0(t(a)).$

\begin{Definition}
    Let $F:S\rightarrow Q$ be a morphism of quivers along with a map $v: S_0 \rightarrow \mathbb{Z}$ which is injective on the fibres $\{y \in S_0: F_0(y)=x\}.$ We define \[\rho_{F,v}: \Rep(S,\mathbb{C}) \rightarrow \Mod \limmauro Q\] by \[\rho_{F,v}(M)_x = \displaystyle\bigoplus_{\substack{ y \in S_0 \\ F_0(y) = x}} M_y,\] \[\rho_{F,v}(M)_a = \displaystyle\bigoplus_{\substack{b \in S_1 \\ F_1(b) = a}} M_b\] and the idempotents are those projecting onto the components of the direct sum $\displaystyle\bigoplus_{F(y) = x} M_y$ and their weights are determined by $v$. The injection \[\displaystyle\bigoplus_{y \in S_0} \Hom_{\mathbb{C}} (M_y, N_y) \hookrightarrow \displaystyle\bigoplus_{x \in Q_0} \Hom_{\mathbb{C}}(\displaystyle\bigoplus_{F(y) = x} M_y, \displaystyle\bigoplus_{F(y) = x} N_y)\] induces an injection \[\rho_{F,v}: \Hom_S(M, N) \hookrightarrow \Hom_{\limmauro{Q}}(\rho_{F,v}(M), \rho_{F,v}(N)).\] For a morphism $f \in \Hom_S(M,N)$ this defines the functor as $\rho_{F,v}(f) \in \Hom_{\limmauro{Q}}(\rho_{F,v}(M), \rho_{F,v}(N)).$ 
\end{Definition}

\begin{Lemma}
    The functor $\rho_{F,v}:\Rep(S,\mathbb{C}) \rightarrow \Mod \limmauro Q$ is an embedding.
\end{Lemma}

\begin{proof}
    This follows from Lemma \ref{embed morph}.
\end{proof}

The following diagram commutes:
\[\begin{tikzcd}
{\Rep(S,\mathbb{C})} &&&& {\crep} \\
\\
&& {\Mod \hat{W}(Q)}
\arrow["{F_*}", from=1-1, to=1-5]
\arrow["{\rho_{F,v}}"', from=1-1, to=3-3]
\arrow["R"', from=3-3, to=1-5]
\end{tikzcd}\]

An object $N \in \Mod \limmauro Q$ gives rise to $\gamma(N)$ and a map $v: \Gamma(N)_0 \rightarrow \mathbb{Z}$ which sends a vertex to the weights of the corresponding non-zero idempotent. Recall that an arrow $a \in \Gamma(N)_1$ is defined by a non-zero map $n^{t(b)}_j N_b n^{s(b)}_i$ for some $b \in Q_1.$ By sending each $a \in \Gamma(N)_1$ to the corresponding arrow $b \in Q_1$ the map $\delta_N$ extends to a morphism of quivers \[\delta_N : \Gamma(N) \rightarrow Q.\] By the definition of $\gamma$ we see \[\rho_{\delta_N,v}(\gamma(N)) = N.\] 

\begin{Lemma} \label{morphism embedding}
    Given a pair $F,v$ as above then $\rho_{F,v}(\Rep(S,\mathbb{C}))$ gives rise to a full subcategory of $\Mod \limmauro{Q}$ and any object in $\Mod \limmauro{Q}$ lies in the essential image of at least one $\rho_{F,v}.$
\end{Lemma}

We see that morphisms of quivers $F\colon S\rightarrow Q$ gives rise to full subcategories of $\Mod \limmauro{Q}$ and $\Mod \limmauro{Q}$ is covered by the subcategories arising from morphisms of quivers. In particular, $\Mod \limmauro{Q}$ allows us to simultaneously study $\operatorname{Rep}(S,\mathbb{C})$ for all morphisms $F\colon S \rightarrow Q$ and a fixed $Q.$ We interpret the next lemma as saying that conversely we can study $\Mod \limmauro{Q}$ by considering all morphisms of quivers to $Q.$

\begin{Lemma}
    For $M,N \in \Mod \limmauro{Q}$ there is a quiver $S,$ representations $X,Y \in \operatorname{Rep}(S,\mathbb{C})$ and a pair $F\colon S \rightarrow Q$ and $v \colon S_0 \rightarrow \mathbb{Z}$ as before such that $M = \rho_{F,v}(X),$ $N = \rho_{F,v}(Y)$ and \[\Hom_S(X,Y) = \Hom_{\limmauro{Q}}(M,N).\]
\end{Lemma}

\begin{proof}
    Recall $\Mod \limmauro{Q}$ is the colimit of the categories $\Mod \mauro{j}{Q}$ with functors $\phi_j: \Mod \mauro{j}{Q} \hookrightarrow \Mod \limmauro{Q}.$ There is a $j \in \mathbb{N}$ such that $X = \phi_j(\tilde{X})$ and $Y = \phi_j(\tilde{Y})$ for $\tilde{X}, \tilde{Y} \in \Mod \mauro{j}{Q}.$ Let $X\oplus Y = \phi_j(\tilde{X}\oplus \tilde{Y}).$ Define $S = \Gamma (M \oplus N).$ Observe $\Gamma(M)$ and $\Gamma(N)$ are naturally subquivers of $S.$ This allows us to consider $\gamma(M)$ and $\gamma(N)$ as objects of $\operatorname{Rep}(S, \mathbb{C}).$ As before we have a morphism of quivers $F\colon S \rightarrow Q$ and $v\colon S_0 \rightarrow \mathbb{Z}$ such that $\rho_{F,v} \colon \operatorname{Rep}(S,\mathbb{C}) \rightarrow \Mod \limmauro{Q}$ is an embedding. The result then follows.
\end{proof}

The above lemma can be easily extended to see that given any finite collection of objects in $\Mod \limmauro{Q}$ all morphisms between them arise from morphisms between representations over a fixed quiver $S$ and a morphism of quivers $F\colon S \rightarrow Q.$ Given a representation $M \in \Rep(S,\mathbb{C})$ of dimension vector $\textbf{e}$ then the dimension vector of $\rho_{F,v}(M)$ is $(\sum_{F_0(y) = x} e_y)_{x \in Q_0}.$ Combined with the embedding \[\rho_{F,v}: \Rep(S,\mathbb{C}) \rightarrow \Mod \limmauro Q\] this then recovers Lemma \ref{splitting up} \[\gr_{\textbf{d}}(N) = \displaystyle\bigsqcup_{\underline{\delta}_{N}(\textbf{e}) = \textbf{d}} \gr_{\textbf{e}}(\gamma(N)).\] For a quiver $S$ denote by $\mathrm{Aut}(S)$ the group of morphisms of quivers which are bijective on $S_0$ and $S_1.$ Given a group $G \subseteq \mathrm{Aut}(S)$ then there is a notion of an orbit quiver $Q \defeq S/G$ as can be found in \cite{MaPe}. This comes with a covering $F: S \rightarrow Q.$ A morphism of quivers $F:S\rightarrow Q$ is called a winding \cite{WCB} if $F_1(a) = F_1(b)$ and $a \neq b$ implies that $s(a) \neq s(b)$ and $t(a) \neq t(b).$

\begin{Lemma}\label{orbit quiver}
    Fix a quiver $S,$ a group $G \subseteq \mathrm{Aut}(S)$ and a map $v: S_0 \rightarrow \mathbb{Z}.$ Suppose that the covering $F: S \rightarrow Q$ is a winding where $Q = S/G$ is the orbit quiver and the induced map \[\begin{aligned}
        w_v : \, & S_1 \rightarrow \mathbb{Z}, \\ & a \mapsto v(t(a)) - v(s(a))
    \end{aligned} \]
    is $G$-invariant. 
    The map $\overline{w}_v : Q_1 \rightarrow \mathbb{Z}$ defines $T \in \mathbb{Z}^{Q_1}$ by $T_a = \overline{w}_v(a)$ and we have \[\mathrm{Im}(\rho_{F,v}) \subseteq \Mod \tlimmauro {Q} {T}.\]
\end{Lemma}

\begin{proof}
    Given $V \in \Rep(S,\mathbb{C})$ then the winding condition ensures that each map in $F_{*}(V) \in \Rep(Q,\mathbb{C})$ sends each weight space to at most one other weight space. The $G$-invariance of the induced function then implies that the weight space it is mapped to is the $T$-homogeneous shift.
\end{proof}

\begin{Remark}
    Under the conditions of Lemma \ref{orbit quiver} we see that given $M \in \Rep(S,\mathbb{C})$ and $N \cong \rho_{F,v}(M)$ then \[\chi(\gr_\textbf{d}(R(N))) = \chi(\gr_\textbf{d}(N)) = \bigsqcup_{F(\textbf{e}) = \textbf{d}} \chi(\gr_\textbf{e}(M)) , \, \forall \textbf{d} \in \mathbb{N}^{Q_0}.\] This can be used to recover Theorem 1.2(c) of \cite{HAU}.
\end{Remark}

\section{Morphism spaces and Euler characteristics}
In this section we will develop a characterisation of nice sequences purely in terms of the category $\Mod \limmauro{Q}.$ Consider $Y \in \Mod \limmauro{Q}.$ For $\textbf{d} \in \mathbb{N}^{Q_0}$ we denote by $\Hom(\textbf{d},Y) \subseteq \bigoplus_{x \in Q_0} \Hom_{\mathbb{C}}(\mathbb{C}^{d_x}, Y_x)$ the space of morphisms $(f_x)_{x \in Q_0}$ such that there is an object $U \in \Mod \limmauro{Q}$ of dimension vector $\textbf{d}$ with $(f_x)_{x \in Q_0}: U \rightarrow Y$ a morphism in $\Mod \limmauro{Q}.$ Given $Y \in \Mod \limmauro{Q}$ of dimension vector $\textbf{m}$ and a subgroup $H \subseteq \mathrm{G}_\textbf{m},$ we will say that $\Hom(-,Y)$ is $H$-stable if $\Hom(\textbf{d},Y) \subseteq \bigoplus_{x \in Q_0} \Hom(\mathbb{C}^{d_x}, Y_x)$ is stable under the action of $H$ for all $\textbf{d} \in \mathbb{N}^{Q_0}.$ Given $Z \in \Mod \limmauro{Q}$ and $t_Z: \cstar \rightarrow \mathrm{G}_\textbf{m}$ the associated morphism we will also denote by $t_Z$ the one parameter subgroup in $\mathrm{G}_\textbf{m}.$ Recall from Proposition \ref{Grass fix} that $\Hom^{0}(\textbf{d},R(N))$ is the space of pairs $(U,f)$ where $U$ is a representation of $Q$ with dimension vector $\textbf{d}$ and $f$ an injective map $f\colon U \hookrightarrow R(N).$

\begin{Lemma}
    For $Y,Z \in \Mod\limmauro{Q}$ then $\gr_\textbf{d}(Y)$ is $t_Z$-stable for all $\textbf{d} \in \mathbb{N}^{Q_0}$ if and only if $\Hom(-,Y)$ is $t_Z$-stable.
\end{Lemma}

\begin{proof}
    If $\Hom(-,Y)$ is $t_Z$-stable, so is $\Hom^{0}(\textbf{d},Y)$ since $t_Z$ acts via invertible maps. Using the quotient $\Hom^{0}(\textbf{d},Y)/\mathrm{G}_\textbf{d} \cong \gr_\textbf{d}(Y)$ we see that this implies the quiver Grassmannian is stable under the $t_Z$-action.

    Suppose the quiver Grassmannians are stable under $t_Z$ it follows that $\Hom^{0}(\textbf{d},Y)$ is too since the actions of $t_Z$ and $\mathrm{G}_\textbf{d}$ commute. By factoring morphisms via their image this implies that $\Hom(-,Y)$ is stable under $t_Z.$
\end{proof}

\begin{Lemma} \label{acyclic stable}
    If Q is acyclic $\gr_\textbf{d}(Y)$ is $t_Z$-stable for all $\textbf{d} \in \mathbb{N}^{Q_0}$ if and only if for any $a\in Q_1$ and $U_{s(a)} \subseteq Y_{s(a)}$ we have \[(t_Z(\lambda)^{-1}Y_at_Z(\lambda))(U_{s(a)}) \subseteq Y_a(U_{s(a)}).\] 
\end{Lemma}

\begin{proof}
    To see that the second condition implies the first we observe that it is equivalent to saying that $(Y_a)(t_Z(\lambda)U_{s(a)}) \subseteq t_Z(\lambda)Y_a(U_{s(a)})$ which says that the $t_Z$-action preserves submodules. Conversely, suppose that $\gr_\textbf{d}(Y)$ is $t_Z$-stable for any $\textbf{d} \in \mathbb{N}^{Q_0}.$ Since $Q$ is acyclic then given any vertex $x \in Q_0,$ $U_x \subseteq Y_x$ and $a \in Q_1$ with $s(a) = x$ there is a submodule of $V \subseteq Y$ with $V_x \cong U_x.$ Explicitly the subspaces $V_x$ are defined by a morphism $f: \left(\mathbb{C}Qe_x\right)^{\dim U_x} \rightarrow R(Y)$ with $Im(e_xf) = V_x.$ The existence of such a morphism follows from the identification $Y_x = \Hom_{\mathbb{C}Q}(\mathbb{C}Qe_x, R(Y))$ (see for example \cite[Prop 2.2.2]{DerWey} and the fact that if $Q$ is acyclic then $\mathbb{C}Qe_x$ is a finite dimensional representation. This submodule being stable under $t_Z$ implies that $(Y_a)(t_Z(\lambda)U_{x}) \subseteq t_Z(\lambda)Y_a(U_{x}).$
\end{proof}

\begin{Proposition} \label{grass and equi}
    Suppose Q is acyclic. For $Z \in \Mod \limmauro{Q}$ and $R(Z) = Y$ then the following are equivalent: 
    \begin{itemize}
        \item $\Hom(-,Y)$ is $t_Z$-stable, 
        \item $\gr_\textbf{d}(Y) \subseteq \prod_{x \in Q_0} \gr(d_x,Y_x)$ is $t_Z$-stable for all $\textbf{d} \in \mathbb{N}^{Q_0},$
        \item $\exists T \in \mathbb{Z}^{Q_1}$ such that $Z \in \Mod\tlimmauro{Q}{T}.$
    \end{itemize} 
\end{Proposition}

\begin{proof}
    We have already shown that the first two conditions are equivalent. 
    We know from Proposition \ref{Grass fix} that the third condition implies the second.  

    Recall that \[t_Z(\lambda)Y_at_Z(\lambda)^{-1} = \sum_{i,k} \lambda^{i-k} z_i^{t(a)} Y_a z_k^{s(a)}.\] To see that the second condition implies the third then suppose that there exists $a \in Q_1$ such that $\sum_{i,k} \lambda^{i-k} z_i^{t(a)} Y_a z_k^{s(a)} \neq \lambda^{T_a} Y_a.$ That is to say there are pairs $(i,k), (j,r)$ such that $i-k \neq j-r,$ $z_i^{t(a)} Y_a z_k^{s(a)} \neq 0,$ and $z_j^{t(a)} Y_a z_r^{s(a)} \neq 0.$ Let $U_{s(a)} \subseteq Y_{s(a)}$ be a line such that $z_i^{t(a)} Y_a z_k^{s(a)}(U_{s(a)}) \neq 0$ and $z_j^{t(a)} Y_a z_r^{s(a)}(U_{s(a)}) \neq 0.$ This implies that for $\lambda \neq \sigma \in \cstar$ then $(t_Z(\lambda)Y_at_Z(\lambda)^{-1})(U_{s(a)})$ and $(t_Z(\sigma)Y_at_Z(\sigma)^{-1})(U_{s(a)})$ are different one-dimensional subspaces. Since $Y_a(U_{s(a)})$ is at most one-dimensional this means there is a $\lambda \in \cstar$ such that $t_Z(\lambda)Y_at_Z(\lambda)^{-1}(U_{s(a)}) \not\subseteq Y_a(U_{s(a)}).$ By Lemma \ref{acyclic stable} this tells us that the second condition implies the third.   
\end{proof}

\begin{Lemma}
    For $Z,Y\in \Mod \limmauro{Q}$ then $t_Z$ acts equivariantly on $\gamma(Y)$ for some $T \in \mathbb{Z}^{\Gamma(Y)_0}$ if and only if $\gr_\textbf{d}(Y)$ is $t_Z$-stable for all $\textbf{d} \in \mathbb{N}^{Q_0}.$
\end{Lemma}

\begin{proof}
    This follows from Lemma \ref{splitting up} and Proposition \ref{grass and equi}. 
\end{proof}

This now gives us another way of characterising equivariant sequences. 

\begin{Theorem} \label{Main theorem}
    Given $(M_0,\cdots,M_r)$ with $M_i \in \Mod \limmauro{Q}$ and $R(M_i) = M$ such that $\Hom(-,M_{i-1})$ is $t_{M_i}$ stable for each $i \in \{1,\cdots,r\}$ and $\Hom(\textbf{d},M_i) \subseteq \Hom(\textbf{d},M_{i-1})$ $\forall \textbf{d} \in \mathbb{N}^{Q_0}$ then \[\chi(\gr_\textbf{d}(M_r)) = \chi(\gr_\textbf{d}(M_0)),\, \forall \textbf{d} \in \mathbb{N}^{Q_0}.\]
\end{Theorem}

We have now developed a dictionary to move between the following data:
\[\text{Nice sequences} \leftrightarrow \text{Equivariant sequences} \leftrightarrow \text{Sequences inducing Hom stability}.\] It would be interesting to see if this final type of data can be given a categorical interpretation in terms of the functor $\Hom(-,M_i).$

\section{Examples}
We aim to demonstrate some of the results on Euler characteristics through two examples.

\begin{Example}
    We first give an example with a primitive equivariant lift. 
    Consider the representation $M =$ 
    \[\begin{tikzcd}
	{\mathbb{C}^3} &&& {\mathbb{C}^3}
	\arrow["{f_1}", shift left=4, from=1-1, to=1-4]
	\arrow["{f_3}"', shift right=4, color={rgb,255:red,92;green,92;blue,214}, from=1-1, to=1-4]
	\arrow["{f_2}"{description}, color={rgb,255:red,214;green,92;blue,92}, from=1-1, to=1-4]
    \end{tikzcd}\]
    where \[f_1 = \left(\begin{matrix} 1 & 0 & 0 \\ 0 & 1 & 0 \\ 0 & 0 & 0 \end{matrix}\right),\] \[f_2 = \left(\begin{matrix} 0 & 0 & 0 \\ 1 & 0 & 0 \\ 0 & 0 & 1\end{matrix}\right),\] \[f_3 = \left(\begin{matrix} 0 & -1 & 0 \\ 0 & 0 & 0 \\ 1 & 0 & 0\end{matrix}\right).\] Let $\textbf{d} = (1,2).$ By fixing an isomorphism $\mathbb{P}^2 \cong \gr(2, 3)$ we obtain equations for this quiver Grassmannian \[\gr_{\textbf{d}}(M) \subseteq \mathbb{P}^{2}_{[x:y:z]} \times \mathbb{P}^{2}_{[a:b:c]},\] by \[ax + by = 0, bx + cz = 0,cx - ay = 0.\] We observe then that the map \[t^{'}(\lambda) = \left(\left(\begin{matrix} 1 & 0 & 0 \\ 0 & \lambda & 0 \\ 0 & 0 & \lambda^{-2} \end{matrix}\right), \left(\begin{matrix} 1 & 0 & 0 \\ 0 & \lambda^{-1} & 0 \\ 0 & 0 & \lambda \end{matrix}\right)\right)\] induces an action on $\mathbb{P}^{2}_{[x:y:z]} \times \mathbb{P}^{2}_{[a:b:c]}$ under which $\gr_{\textbf{d}}(M)$ is stable. Since we passed to a quotient at the second vertex the morphism we will want to consider for equivariance is \[t(\lambda) = \left(\left(\begin{matrix} 1 & 0 & 0 \\ 0 & \lambda & 0 \\ 0 & 0 & \lambda^{-2} \end{matrix}\right), \left(\begin{matrix} 1 & 0 & 0 \\ 0 & \lambda & 0 \\ 0 & 0 & \lambda^{-1} \end{matrix}\right)\right).\] We now check that $M(t)$ is a primitive equivariant lift. We display $\Gamma(M(t)):$ 
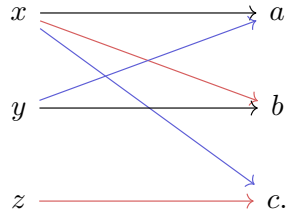
    \[
    \begin{tikzcd}
	x &&& a \\
	y &&& b \\
	z &&& c.
	\arrow[from=1-1, to=1-4]
	\arrow[color={rgb,255:red,214;green,92;blue,92}, from=1-1, to=2-4]
	\arrow[color={rgb,255:red,92;green,92;blue,214}, from=1-1, to=3-4]
	\arrow[color={rgb,255:red,92;green,92;blue,214}, from=2-1, to=1-4]
	\arrow[from=2-1, to=2-4]
	\arrow[color={rgb,255:red,214;green,92;blue,92}, from=3-1, to=3-4]
    \end{tikzcd}
    \]

    The choice of $t$ defines the grading $\partial$ with $\partial(x) = 0, \partial(y) = 1, \partial(z) = -2, \partial(a) = 1, \partial(b) = 1, \partial(c) = -1.$ The black arrows correspond to $f_1,$ the red to $f_2$ and the blue to $f_3.$ We see from $\Gamma(M(t))$ that the black arrows preserve the grading. The red arrows shift the grading by 1 and the blue arrows shift the grading by -1. This tells us that for $T = (0,1,-1)$ we have $M(t) \in \Mod \tlimmauro Q T.$ We can use it to compute $\gr_{\textbf{d}}(M(t)).$ It is a zero dimensional variety where $|\gr_{\textbf{d}}(M(t))| = 3,$ i.e. it is the disjoint union of three points. We should check that we actually believe that $$\chi(\gr_{\textbf{d}}(M)) = 3$$ which we can do through either projection. If we project onto the first vertex we get a map $$p_1: \gr_{\textbf{d}}(M) \twoheadrightarrow V(y^2z - x^3) \subseteq \mathbb{P}^2$$ which is bijective over the smooth locus and has fibre $\mathbb{P}^1$ over the isolated singularity. If we project onto the second vertex we see that $$p_2: \gr_{\textbf{d}}(M) \twoheadrightarrow V(a^2b + b^2c) \subset \mathbb{P}^2$$ is a normalisation. The computation is confirmed by the fact that the Euler characteristic of the cuspidal cubic is two or that the Euler characteristic of the second cubic is three. 

    We can deform this family to give an example of Lemma \ref{no zeros}. Consider $M_r = $

     \[\begin{tikzcd}
	{\mathbb{C}^3} &&& {\mathbb{C}^3}
	\arrow["{f_1}", shift left=4, from=1-1, to=1-4]
	\arrow["{f_3(r)}"', shift right=4, color={rgb,255:red,92;green,92;blue,214}, from=1-1, to=1-4]
	\arrow["{f_2}"{description}, color={rgb,255:red,214;green,92;blue,92}, from=1-1, to=1-4]
    \end{tikzcd}
    \]

    where we have deformed to $$f_3(r) = \left(\begin{matrix} 0 & -1 & 0 \\ 0 & 0 & r \\ 1 & 0 & 0\end{matrix}\right).$$ Notice that $M_0 = M.$ When $r \neq 0$ then $\gr_{\textbf{d}}(M_r)$ is isomorphic to the smooth cubic $$V(Y^2Z - X^3 - rXZ^2) \subset \mathbb{P}^2.$$ Since it is a smooth cubic curve it has Euler characteristic zero and we can likewise see that the extra arrow provides the obstruction to a primitive equivariant lift \[
    \begin{tikzcd}
	x &&& a \\
	y &&& b \\
	z &&& c.
	\arrow[from=1-1, to=1-4]
	\arrow[color={rgb,255:red,214;green,92;blue,92}, from=1-1, to=2-4]
	\arrow[color={rgb,255:red,92;green,92;blue,214}, from=1-1, to=3-4]
	\arrow[color={rgb,255:red,92;green,92;blue,214}, from=2-1, to=1-4]
	\arrow[from=2-1, to=2-4]
	\arrow["{r\neq0}"{description}, color={rgb,255:red,92;green,92;blue,214}, from=3-1, to=2-4]
	\arrow[color={rgb,255:red,214;green,92;blue,92}, from=3-1, to=3-4]
    \end{tikzcd}
    \] 
\end{Example}

\begin{Example} \label{lim ex}
    We give an example where two steps are definitely needed in Proposition \ref{haupt}. We can consider Example 4.7 from \cite{JS}. Let $M =$
    \[\begin{tikzcd}
	{\mathbb{C}^5}
	\arrow["{f_1}", from=1-1, to=1-1, loop, in=145, out=215, distance=10mm]
	\arrow["{f_2}", from=1-1, to=1-1, loop, in=325, out=35, distance=10mm]
\end{tikzcd}\]
    where $$f_1 = \left(\begin{matrix} 0 & 0 & 0 & 0 & 0 \\ 1 & 0 & 0 & 0 & 0 \\ 0 & 0 & 0 & 1 & 0 \\ 0 & 0 & 0 & 0 & 0 \\ 0 & 0 & 0 & 0 & 0\end{matrix}\right),$$ $$f_2 = \left(\begin{matrix} 0 & 0 & 0 & 0 & 0 \\ 0 & 0 & 0 & 0 & 0 \\ 0 & 1 & 0 & 0 & 0 \\ 0 & 0 & 0 & 0 & 1 \\ 0 & 0 & 0 & 0 & 0\end{matrix}\right).$$ This module doesn't have a primitive equivariant lift however by first using the equivariant lift given by $$t(\lambda) =  \left(\begin{matrix} \lambda^3 & 0 & 0 & 0 & 0 \\ 0 & \lambda^2 & 0 & 0 & 0 \\ 0 & 0 & 1 & 0 & 0 \\ 0 & 0 & 0 & \lambda & 0 \\ 0 & 0 & 0 & 0 & \lambda^3 \end{matrix}\right)$$ we get $\gamma(M(t)) = $

    \[\begin{tikzcd}
	& {\mathbb{C}^2} \\
	{\mathbb{C}} && {\mathbb{C}} \\
	& {\mathbb{C}}
	\arrow["{\rowveca}"', from=1-2, to=2-1]
	\arrow["{\rowvecb}", from=1-2, to=2-3]
	\arrow["1"', from=2-1, to=3-2]
	\arrow["1", from=2-3, to=3-2]
    \end{tikzcd}\]

    The limit of the associated equivariant sequence is then given by attaching an idempotent to M for each element of the standard basis of M which projects onto the corresponding basis element. The remaining idempotents are then zero.   
\end{Example}

\section{Filtrations and representation varieties}
Recall the representation variety \[\mathrm{E}^{Q}_\textbf{m} = \bigoplus_{a \in Q_1} \Hom(\mathbb{C}^{m_{s(a)}},\mathbb{C}^{m_{t(a)}})\] of a quiver Q for $\textbf{m} \in \mathbb{N}^{Q_0}.$ As before, the group $\mathrm{G}_\textbf{m} = \displaystyle\prod_{x \in Q_0} \mathrm{GL}(\mathbb{C}^{m_x})$ acts on the representation variety. The orbits are in bijection with isoclasses of representations of dimension vector $\textbf{m}.$ Let $Q$ be an equioriented quiver of type A$_n.$ That is to say, the vertices of $Q$ are given by $Q_0 = \{1, \dots,n\}$ and the arrows are given by $a_{i}: i \rightarrow i+1$ for $i \in \{1,\dots,n-1\}.$ The doubled quiver $\overline{Q}$ has the same vertices as $Q$ and arrows $\{a \in Q_1\} \cup \{a^*: t(a) \rightarrow s(a), \, \text{ for } a \in Q_1\}.$ The representation variety of $\overline{Q}$ \[\mathrm{E}^{\overline{Q}}_\textbf{m} = \bigoplus_{a \in Q_1} \left(\Hom(\mathbb{C}^{m_{s(a)}}, \mathbb{C}^{m_{t(a)}}) \oplus \Hom(\mathbb{C}^{m_{t(a)}}, \mathbb{C}^{m_{s(a)}})\right)\] can be identified with the cotangent space of $\mathrm{E}_\textbf{m}.$ The map \[\mu:\mathrm{E}^{\overline{Q}}_\textbf{m} = \mathrm{T}^*\mathrm{E}_\textbf{m} \rightarrow \oplus_{i \in Q_0} \mathfrak{gl}_{m_i}\] given by \[\mu((x_a,x_{a^*})_{a \in Q_1}) = \left(\sum_{\substack{a \in Q_1\\ t(a) = i}} x_{a^*}x_a - \sum_{\substack{a \in Q_1\\ h(a) = i}} x_ax_{a^*}\right)_{i \in Q_0}\] is called a moment map. There is a natural $\cstar$-action on $\mathrm{T}^*\mathrm{E}_\textbf{m}$ which acts by $\lambda^{-1}$ on the fibres of $\mathrm{T}^*\mathrm{E}_\textbf{m} \rightarrow \mathrm{E}_\textbf{m}.$ This is the same as letting $T = (0,-1,0,-1,\dotsc,0,-1) \in \mathbb{Z}^{\overline{Q}_1}$ with the zeros associated to $a \in Q_1$ and the $-1$s associated to the $a^*.$ Letting $\cstar$ act on $\oplus_{i \in Q_0} \mathfrak{gl}_{m_i}$ via multiplication by $\lambda^{-1}$ we see that $\mu$ is $\cstar$-equivariant. In particular, the fibres are stable under this $\cstar$-action. The subspace $\Lambda_\textbf{m} = \mu^{-1}(0)$ is the representation variety of dimension vector $\textbf{m}$ for the preprojective algebra of $Q.$   

In \cite[Section 9]{GLSsemi}, Geiss, Leclerc, and Schr\"oer study a Galois covering of the type A$_n$ preprojective algebra $\Pi(A_n).$ Let $\tilde{Q}_n$ be the quiver with vertices $\{i_j: i \in \{1,\cdots,n\},\,j\in\mathbb{Z}\}$ and arrows $\alpha_{i,j} : i_j \rightarrow (i+1)_j$ and $\alpha^*_{i,j}: (i+1)_j \rightarrow i_{j-1}.$ The relations $J_n$ are generated by \[\alpha^*_{1,j}\alpha_{1,j},\,\, \alpha_{n-1,j-1}\alpha^*_{n-1,j}, \, \, \alpha^*_{i,j}\alpha_{i,j} - \alpha_{i-1,j-1}\alpha^*_{i-1,j}.\] The natural morphism of quivers $F: \tilde{Q}_n \rightarrow \overline{Q}$ induces the pushforward \[F_*:\Mod \mathbb{C}\tilde{Q}_n/J_n \rightarrow \Mod \Pi(A_n).\] This is a Galois covering and is used in \cite{GLSsemi} to study the irreducible components of the representation varieties $\Lambda_\textbf{m}$ when $n \leq 5.$ Let $v: (\tilde{Q}_n)_0 \rightarrow \mathbb{Z}$ be given by $v(i_j) = j.$ This is injective on the sets $\{y \in (\tilde{Q}_n)_0:F_0(y) = i\}.$ 

Recall from Definition \ref{Alg mauro} that for $\Pi(A_n)$ we have the categories $\Mod \limmauro {\Pi(A_n)}$ and $\Mod \tlimmauro{\Pi(A_n)}{T}.$ These are the full subcategories of $\Mod \limmauro{\overline{Q}}$ and $\Mod \tlimmauro{\overline{Q}}{T}$ respectively consisting of objects $N$ such that $R(N) \in \Mod \Pi(A_n).$ It follows from Lemma \ref{orbit quiver} that there is an embedding \[\rho_{F,v}:\Mod \mathbb{C}\tilde{Q}_n/J_n \rightarrow \Mod \limmauro {\Pi(A_n)}.\] Technically Lemma \ref{orbit quiver} gives us an embedding to $\Mod \limmauro{\overline{Q}}$ however it is easy to check that it embeds into the full subcategory $\Mod \limmauro {\Pi(A_n)}.$

\begin{Proposition}
    The functor $\rho_{F,v}$ induces an equivalence of categories \[\Mod \mathbb{C}\tilde{Q}_n/J_n  \rightarrow  \Mod \tlimmauro{\Pi(A_n)}{T}\] for $T = (0,-1,0,-1,\dotsc,0,-1).$
\end{Proposition}

\begin{proof}
    We will show the embedding $\rho_{F,v}$ factors via $\Mod \tlimmauro{\Pi(A_n)}{T}$ and is essentially surjective on $\Mod \tlimmauro{\Pi(A_n)}{T}.$ As explained in \cite{GLSsemi} the map $F: \tilde{Q}_n \rightarrow \overline{Q}$ is quotienting by a $\mathbb{Z}$-action. This action is shifting the indices on the vertices. Clearly $v(t(a)) - v(s(a))$ is invariant under this shift and the morphism $F$ is a winding. By Lemma \ref{orbit quiver} we see that $\mathrm{Im}(\rho_{F,v}) \subseteq \Mod \tlimmauro{\Pi(A_n)}{T}$ for $T = (0,-1,\dotsc,0,-1).$ Consider $M \in \Mod \tlimmauro{\Pi(A_n)}{T}.$ Since $M$ satisfies $T$-homogeneity then we know that $\Gamma(M)$ embeds as a subquiver of $\tilde{Q}_n$ and satisfies the necessary relations given by $J_n$ since these are induced by the preprojective relations. This allows us to consider $\gamma(M)$ as an object of $\Mod \mathbb{C}\tilde{Q}_n/J_n.$ Since $\rho_{F,v}(\gamma(M)) = M$ we have that the functor \[\Mod \mathbb{C}\tilde{Q}_n/J_n  \rightarrow  \Mod \tlimmauro{\Pi(A_n)}{T}\] is essentially surjective.
\end{proof}

\begin{Remark}
    We see by Theorem \ref{equivariant} that the category of representations of the Galois cover is equivalent to the category of $\cstar$-equivariant $\Pi(A_n)$ modules with $\cstar$ acting on the fibres of the cotangent bundle as above. 
\end{Remark}

Let $\textbf{j}$ be a reduced expression for the longest element in $S_{n+1}.$ Let $l_n = n(n+1)/2.$ We will think of $\textbf{j}$ as a function \[\textbf{j}:\{1,\cdots,l_n\} \rightarrow \{1,\cdots,n\}.\] Furthermore, we define \[\textbf{j}_x = \{i:\textbf{j}(i) = x\}.\]

Let $\overline{Q}^\textbf{j}$ denote the quiver with $\overline{Q}^\textbf{j}_0 = Q_0$ and $\overline{Q}^\textbf{j}_1 = \overline{Q}_1 \cup \{\ell_i^x, \, x \in Q_0, \, i \in \textbf{j}_x\}.$

\begin{Definition}
    Define
    \[\filt = \mathbb{C}\overline{Q}^\textbf{j}/I\] where $I$ is generated by the preprojective relations on $Q$ as well as
    \[\ell_i^x \ell_j^x, \, i \neq j,\, x \in [n],\]
    \[\sum_i \ell_i^x - e_x,\, x \in [n],\]
    \[\ell_j^{t(a)} a \ell_i^{s(a)},\, a \in \overline{Q}_1, j \geq i.\]
\end{Definition}

Note that this is a quotient of the algebra $\mauro{l_n}{\Pi(A_n)}.$ This means we can consider $\Mod \filt$ as a full subcategory of $\Mod \mauro{l_n}{\Pi(A_n)}$ and in turn a full subcategory of $\Mod \mathbb{C}\overline{Q}^{\textbf{j}}.$ 

\begin{Definition}
    Consider an object $M \in \Mod \filt$ of dimension vector $\textbf{m} \in \mathbb{N}^n.$ Via the embedding $\Mod \filt \subset \Mod \mathbb{C}\overline{Q}^{\textbf{j}}$ we can send the object $M$ to a point in $\mathrm{E}_\textbf{m}(\overline{Q}^{\textbf{j}}).$ We denote by $\mathrm{E}_\textbf{m}(\filt)$ the subspace of $\mathrm{E}_\textbf{m}(\overline{Q}^{\textbf{j}})$ consisting of all points arising from objects of $\Mod \filt$ through this process. 
\end{Definition}

We can think of $\mathrm{E}_\textbf{m}(\filt)$ as a representation variety for $\filt.$

\begin{Example}
    When $n = 2$ and $\textbf{j} = (1,2,1)$ the algebra $\filt$ is given by 
    \[\begin{tikzcd}[ampersand replacement=\&]
    {x} \& {y}
    \arrow[loop above, from=1-1, to=1-1, "{\ell_{1}^x}"]
    \arrow[loop below, from=1-1, to=1-1, "{\ell_{3}^x}"]
    \arrow[shift left = 1, from=1-1, to=1-2, "{a}"]
    \arrow[shift left = 1, from=1-2, to=1-1, "{b}"]
    \arrow[loop right, from=1-2, to=1-2, "{\ell_{2}^y}"]
    \end{tikzcd}\] and the relations \[ab - ba, \, \ell_2^y - e_y, \, \ell_1^x \ell_3^x, \, \ell_3^x \ell_1^x,\, \ell_1^x + \ell_3^x - e_x ,\, \ell^y_2 a \ell_1^x, \,  \ell_3^x b \ell_2^y.\]
\end{Example}

The functor $R: \Mod \mauro{l_n}{\Pi(A_n)} \rightarrow \Mod \Pi(A_n)$ induces a map $R_{\star}: \mathrm{E}_\textbf{m}(\filt) \rightarrow \Lambda_{\textbf{m}}.$ Given $\textbf{i} = (i_1,\cdots,i_t)\in [n]^t,$ $\textbf{a} = (a_1,\cdots,a_t) \in \mathbb{N}^t,$ and $M \in \Mod \Pi(A_n)$ denote by $\Phi_{\textbf{i}^{\textbf{a}},\textbf{m}}$ the variety of flags of submodules 
\[\{0 = M^0 \subseteq M^1\subseteq \cdots \subseteq M^t = M: M^j/M^{j-1} \cong S_{i_j}^{a_j}, \, M \in \Lambda_\textbf{m}\}.\] This slightly differs from the notation used in \cite{GLSr}.

We consider the commuting diagram 
\[\begin{tikzcd}[ampersand replacement=\&]
	{\mathrm{E}_\textbf{m}(\filt)} \&\& {\displaystyle\bigcup_{\textbf{a}} \Phi_{\textbf{j}^{\textbf{a}}, \textbf{m}}} \\
	\\
	\&\& {\Lambda_\textbf{m}}
	\arrow["{p_1}", from=1-1, to=1-3]
	\arrow["{R_\star}"', from=1-1, to=3-3]
	\arrow["{p_2}", from=1-3, to=3-3]
\end{tikzcd}\]

where $p_2$ is the natural projection onto $\Lambda_\textbf{m}$ and $p_1: \mathrm{E}_\textbf{m}(\filt) \rightarrow {\displaystyle\bigcup_{\textbf{a}} \Phi_{\textbf{j}^{\textbf{a}}, \textbf{m}}}$ sends $M$ to $R(M)$ along with the flag given by \[(M^i)_x = \mathrm{Im}(\sum_{j \leq i} m_j^x)\] at each vertex $x \in \{1, \dots, n\}.$  

\begin{Lemma} \label{idem grad}
    Given a vector space $V$ then a choice of idempotents $e_1, \dots, e_t$ with $e_ie_j = 0$ if $i \neq j$ and $\mathrm{Id}_V = \displaystyle\sum_{1 \leq i \leq t} e_i$ is equivalent a choice of grading $V = \displaystyle\bigoplus_{1 \leq i \leq t} V_i.$ 
\end{Lemma}

\begin{proof}
    This is a standard result but we include it for completeness. Given such a choice of idempotents then defining $V_i = \mathrm{Im}(e_i)$ defines a grading on $V.$ Conversely given a grading $V = \oplus_i V_i$ we let $e_i$ be the idempotent projecting onto the component $V_i.$ To see that these are uniquely determined it suffices to show that for $(e_i)_{1 \leq i \leq t}$ and $(f_i)_{1 \leq i \leq t}$ then $e_i = f_i$ $\forall  \leq i \leq t$ if and only if $\mathrm{Im}(e_i) = \mathrm{Im}(f_i)$ $\forall 1 \leq i \leq t.$ Clearly if $e_i = f_i$ then $\mathrm{Im}(e_i) = \mathrm{Im}(f_i).$ Therefore suppose that $\mathrm{Im}(e_i) = \mathrm{Im}(f_i)$ for $1\leq i \leq t.$ It immediately follows that $e_jf_i = 0$ if $j \neq i.$ Therefore \[\begin{aligned} e_i &= e_i \mathrm{Id}_V \\ &= e_i \left(\displaystyle\sum_{1 \leq i \leq t} f_i \right) \\ &= e_if_i\end{aligned}\] so $e_if_i = e_i = e_i e_i.$ Now \[\begin{aligned} e_i - f_i &= \mathrm{Id}_V (e_i-f_i) \\ &= e_ie_i - e_if_i \\ &= 0. \end{aligned}\] This completes the proof.
\end{proof}

\begin{Lemma} \label{affine fibre}
    The fibres of $p_1$ are affine spaces with constant dimension over $\Phi_{\textbf{j}^{\textbf{a}}, \textbf{m}}$ for each $\textbf{a} \in \mathbb{N}^{l_n}$
\end{Lemma}

\begin{proof}
    It is a standard homological fact that given an exact sequence of vector spaces \[U \hookrightarrow V \twoheadrightarrow V/U\] then the space of splittings of this exact sequence can be identified with $\Hom_{\mathbb{C}}(V/U, U).$ This follows from applying $\Hom_{\mathbb{C}}(-,U)$ to the exact sequence to get \[\Hom_{\mathbb{C}}(V/U,U) \hookrightarrow \Hom_{\mathbb{C}}(V,U) \twoheadrightarrow \Hom_{\mathbb{C}}(U,U).\] The space of morphisms in $\Hom_{\mathbb{C}}(V,U)$ which get mapped to the identity in $\Hom_{\mathbb{C}}(U,U)$ can then be identified with $\Hom_{\mathbb{C}}(V/U,U).$ This can then be inductively extended to show that the space of gradings on a vector space which induce a filtration $V^1 \subset V^2 \subset \cdots \subset V^r = V$ can be identified with \[\displaystyle\bigoplus_{j}\Hom_{\mathbb{C}}(V^j/V^{j-1}, V^{j-1}).\] Note that all of these identifications are not canonical as one needs to choose a basepoint.
    For $N \in \filt$ the choice of idempotents at vertex $x \in Q_0$ uniquely determines a grading of the vector space $N_x$ as described above in Lemma \ref{idem grad}. Under this identification we can then apply the standard grading result to identify the fibre $p_1^{-1}(\underline{F})$ for $\underline{F} = (0 = M^0 \subseteq M^1\subseteq \cdots \subseteq M^t = M, M) \in \Phi_{\textbf{i}^{\textbf{a}},\textbf{m}}$ with \[\displaystyle\prod_{x \in Q_0} \displaystyle\bigoplus_{j} \Hom_{\mathbb{C}}((M^j/M^{j-1})_x, (M^{j-1})_x).\] The dimension of these affine spaces are fully determined by the dimensions of the $M^i/M^{i-1}$ which are determined by $\textbf{a}.$ 
\end{proof}

Since all of the fibres are connected then the connected components of ${\mathrm{E}_\textbf{m}(\filt)}$ coincide with those of ${\displaystyle\bigcup_{\textbf{a}} \Phi_{\textbf{j}^{\textbf{a}}, \textbf{m}}}.$ 
Each $\Lambda_\textbf{m}$ is connected since the semisimple module of that dimension vector is in the closure of any module of that dimension vector. Since $\mathrm{G}_\textbf{m}$ has a transitive action on the space of flags of type $\mathbf{a}$ this implies that each $\Phi_{\textbf{j}^{\textbf{a}}, \textbf{m}}$ is connected. Therefore the connected components of ${\mathrm{E}_\textbf{m}(\filt)}$ can be labeled by $\textbf{a} \in \mathbb{N}^{l_n}.$ Given C, a connected component of ${\mathrm{E}_\textbf{m}(\filt)},$ we denote by $r(C) \in \mathbb{N}^{l_n}$ the associated label of the component $\Phi_{\textbf{j}^{r(C)}, \textbf{m}}.$

Given a subset $A \subset X$ we denote by $\mathds{1}_A : X \rightarrow \mathbb{C}$ the function defined by $\mathds{1}_A(x) = 1$ if $x \in A$ and $0$ otherwise. We refer to them as characteristic functions. A constructible function is a finite sum $g = \sum_i a_i \mathds{1}_{Z_i},$ $a_i \in \mathbb{C}$ where $Z_i$ is locally closed for each $i.$ Let $\widetilde{\mathcal{H}}_\textbf{m}$ denote the space of $\mathrm{G}_\textbf{m}$-equivariant constructible functions on $\mathrm{E}_\textbf{m}(\filt).$ Let $\widetilde{\mathcal{M}}_\textbf{m}$ denote the space of $\mathrm{G}_\textbf{m}$-equivariant constructible functions on $\Lambda_\textbf{m}.$ For a constructible function, $g = \sum_i a_i \mathds{1}_{Z_i}$ define \[\int_A g \, d\chi = \sum_i a_i \chi(A \cap Z_i).\] We define the following map 
\[\xi:  \widetilde{\mathcal{H}}_\textbf{m} \rightarrow \widetilde{\mathcal{M}}_\textbf{m}\] by \[\xi(f)(M) = \displaystyle\int_{R_{\star}^{-1}(M)} f \, d\chi.\] 

There is a product on $\oplus_{\textbf{m}} \widetilde{\mathcal{M}}_\textbf{m}$ given by $(f \star g)(M) = \int_{U \subseteq M} f(U)g(M/U) \, d\chi.$ The algebra $\mathcal{M} \subseteq \oplus_{\textbf{m}} \widetilde{\mathcal{M}}_\textbf{m}$ is the subalgebra generated by the characteristic functions on the representation varieties associated to the one-dimensional representations at vertex $i.$ We denote them by $\mathds{1}_i,\,i \in \{1,\dots,n\}.$ Let $\mathcal{M}_{\textbf{m}} = \mathcal{M} \cap \widetilde{\mathcal{M}_\textbf{m}}.$

We consider $\mathcal{H}_\textbf{m},$ the subspace of $\widetilde{\mathcal{H}}_\textbf{m}$ spanned by the characteristic functions on connected components. This is the space of functions of the form 
\[\mathcal{H}_\textbf{m} = \{f \in \widetilde{\mathcal{H}}_\textbf{m}(A): f = \sum_i a_i \mathds{1}_{C_i},\, C_i \in \widetilde{\mathrm{C}}_\textbf{m}, a_i \in \mathbb{C}\}\] where $\widetilde{\mathrm{C}}_\textbf{m}$ denotes the set of connected components of $\mathrm{E}_\textbf{m}(\filt).$ 

For $\mathbf{a} \in\mathbb{N}^t$ and $\textbf{j}^{\textbf{a}}$ denote by $\mathds{1}_{\textbf{j}^{\textbf{a}}} \in \mathcal{M}$ the function \[\mathds{1}_{\textbf{j}^{\textbf{a}}} = \underbrace{\mathds{1}_{j_1}\star \dotsb \star \mathds{1}_{j_1}}_{a_1 - \text{times}}  \star \dotsb \star \overbrace{\mathds{1}_{j_{l_n}} \star \dotsb \star \mathds{1}_{j_{l_n}}}^{a_{l_n}-\text{times}}.\]
\begin{Lemma} \label{pushforward}
    Given a connected component $C$ of $\mathrm{E}_\textbf{m}(\filt)$ then \[(\prod_i a_i!)\xi(\mathds{1}_C) = \mathds{1}_{\textbf{j}^{\textbf{a}}}\] where $\textbf{a} = r(C).$ 
\end{Lemma}

\begin{proof}
    Given $M \in \Lambda_{\textbf{m}}$ then 
    \[\begin{aligned}
        \xi(\mathds{1}_C)(M) &= \chi(\{N \in C: R(N) = M\}) \\ &= \chi(R_\star^{-1}(M) \cap C) \\ &= \chi(p_1^{-1}p_2^{-1}(M) \cap C) \\ &= \chi(p_2^{-1}(M)\cap \Phi_{\textbf{j}^{\textbf{a}}, \textbf{m}})
    \end{aligned}\] where ${\textbf{a}} = r(C).$
    In the last step we have used Lemma \ref{affine fibre}. Since $\chi(\mathbb{P}^t) = (t+1)$ we see that $(\prod_i a_i!)\chi(p_2^{-1}(M)\cap \Phi_{\textbf{j}^{\textbf{a}}, \textbf{m}}) = \mathds{1}_{\textbf{j}^\textbf{a}}(M).$     
\end{proof}

\begin{Proposition}
    $\xi(\mathcal{H}_\textbf{m}) = \mathcal{M}_\textbf{m}.$
\end{Proposition}

\begin{proof}
    By the above lemma we know that $\xi(\mathcal{H}_\textbf{m}) \subseteq \mathcal{M}_\textbf{m}.$ We must now show that any function in $\mathcal{M}_\textbf{m}$ is in the span of the $\mathds{1}_{\textbf{j}^\textbf{a}}$ where $\textbf{j}$ is a fixed reduced expression for the longest element in $S_{n+1}.$ Let $\delta_{\textbf{j}^\textbf{a}} \in \mathcal{M}^{*}$ denote the dual element to $\mathds{1}_{\textbf{j}^\textbf{a}} \in \mathcal{M}.$ Suppose $f \in \mathcal{M}$ is not in the span of the $\mathds{1}_{\textbf{j}^\textbf{a}}.$ Then $g = f - \sum_{\textbf{a}} \delta_{\textbf{j}^\textbf{a}}(f)\mathds{1}_{\textbf{j}^\textbf{a}} \in \mathcal{M}$ is non-zero. For $\tau \in \mathcal{M}^{*}$ dual to $g$ then $\tau \neq 0$ but $\tau(\mathds{1}_{\textbf{j}^{\textbf{a}}}) = 0 ,\, \forall \textbf{a} \in \mathbb{N}^{l_n}.$ Therefore it suffices to show that for any $\tau \in \mathcal{M}^{*}$ if \[\tau(\mathds{1}_{\textbf{j}^{\textbf{a}}}) = 0 ,\, \forall \textbf{a} \in \mathbb{N}^{l_n}\] then $\tau = 0.$ We know by \cite{GLSsemi} that $\mathcal{M}^{*}$ is spanned by the functions $\delta_M,\,M \in \Mod \Pi(A_n)$ where $\delta_M(f) = f(M).$ By \cite[Lemma 9.1]{GLSr} and the comments following it, if \[\delta_M(\mathds{1}_{\textbf{j}^{\textbf{a}}}) = 0 ,\, \forall \textbf{a} \in \mathbb{N}^{l_n}\] then $\delta_M = 0.$
\end{proof}

Given $y \in \mathbb{C}$ let $x_i(y) = \mathrm{Id} + yE_{i,i+1}\in N_{-}$ for $N_{-}$ the big affine cell in the type $A_n$ flag variety (see for example \cite{GLSsemi}). The set \[\{x_{j_1}(y_1)\cdots x_{j_{l_n}}(y_{l_n}): (y_1,\dotsc,y_{l_n}) \in \mathbb{C}^{l_n}\}\] is dense in $N_{-}$ when $\textbf{j} = j_1\cdots j_{l_n}$ is a reduced expression for $w_0.$ In \cite{GLSr}, Geiss, Leclerc, and Schr\"oer introduced a function $\varphi:\Mod\Pi(A_n) \rightarrow\mathbb{C}[N_-].$ For $\textbf{a} \in \mathbb{N}^{l_n}$ we denote by $y^{\textbf{a}} = \prod_{i=1}^{l_n} y_{i}^{a_t} \in \mathbb{C}$ where $(y_1,\dotsc,y_{l_n}) \in \mathbb{C}^{l_n}.$ Let $\widetilde{C}$ be the set of connected components of $\bigcup_{\textbf{m} \in \mathbb{N}^{n}} \mathrm{E}_\textbf{m}(\filt).$ 

\begin{Proposition}
    \[\varphi(x_{j_1}(y_1)\cdots x_{j_{l_n}}(y_{l_n})) = \sum_{C \in \widetilde{C}} y^{r(C)}\xi(\mathds{1}_C), \, \forall y = (y_1,\cdots,y_{l_n}) \in \mathbb{C}^{l_n}\] considered as functions on $\Mod \Pi(A_n).$
\end{Proposition}

\begin{proof}
    This follows from Lemma \ref{pushforward} and \cite[Lemma 9.1]{GLSr}.
\end{proof}

Notice that we have constructed a surjection of graded vector spaces \[\xi: \bigoplus_\textbf{m} \mathcal{H}_\textbf{m} \twoheadrightarrow \bigoplus_\textbf{m} \mathcal{M}_\textbf{m} = \mathcal{M} \cong U(\mathfrak{n}_+).\] The final isomorphism is coming from \cite{Lus2}. Both the left and right hand side have bases indexed by elements of $\mathbb{N}^{l_n}$ coming from $r(C)$ on the left and the PBW basis on the right. This map is however not injective. For example there are many functions which will map to $\mathds{1}_i$ when the index $i$ occurs multiple times in $\textbf{j}.$

\begin{Remark}
    We note that the surjection $\xi: \mathcal{H} \twoheadrightarrow \mathcal{M}$ is not an algebra morphism. We can give $\widetilde{\mathcal{H}}$ the structure of an algebra by defining \[(f \star g)(M) = \int_{U \subseteq M} f(U)g(M/U) \, d\chi\] as before, however $\mathcal{H}$ is not closed under this product. In particular, we can easily observe that \[\operatorname{supp}(\mathds{1}_{C_1} \star \mathds{1}_{C_2}) \subseteq C_3\] where $C_3$ is the unique connected component of $\mathrm{E}(\filt)$ such that $r(C_3) = r(C_1) + r(C_2)$ but typically $\mathds{1}_{C_1} \star \mathds{1}_{C_2}$ is not a multiple of $\mathds{1}_{C_3}.$ 
\end{Remark}

\begin{Example}
    Consider 
    \[\begin{tikzcd}
	1 & 2
	\arrow["a", shift left, from=1-1, to=1-2]
	\arrow["b", shift left, from=1-2, to=1-1]
    \end{tikzcd}\] with $ab = ba = 0.$ This is the type A$_2$ preprojective algebra. There are 4 indecomposable representations of $\Pi(A_2)$ given by the two simple representations $S_1,S_2$ and the two projective-injectives $P_1,P_2.$ Let $\textbf{j} = (1,2,1).$ This is a reduced expression for the longest element in $S_3.$ Consider $\textbf{m} = (1,2).$ The variety $\mathrm{E}_\textbf{m}(\filt)$ has two connected components given by $r(C_1) = (1,2,0)$ and $r(C_2) = (0,2,1).$ The representation variety $\Lambda_\textbf{m}$ of the preprojective algebra has two irreducible components. One component $Z_1 \subseteq \Lambda_\textbf{m}$ where $Z_1 = \overline{\mathrm{G}_\textbf{m} \cdot (P_1\oplus S_2)}.$ The second component $Z_2 \subseteq \Lambda_\textbf{m}$ where $Z_2 = \overline{\mathrm{G}_\textbf{m} \cdot (P_2\oplus S_2)}.$ It is easy to compute $\xi(\mathds{1}_{C_1}) = \mathds{1}_{Z_1}$ and $\xi(\mathds{1}_{C_2}) = \mathds{1}_{Z_2}.$ 

    Conversely, given $\textbf{m} = (2,1)$ we have three components of $\mathrm{E}_\textbf{m}(\filt).$ These are given by $r(C_1) = (2,1,0),$ $r(C_2) = (0,1,2),$ and $r(C_3) = (1,1,1).$ Once again we have two components of $\Lambda_\textbf{m}$ given by $Z_1 = \overline{\mathrm{G}_\textbf{m} \cdot (P_1\oplus S_1)}$ and $Z_2 = \overline{\mathrm{G}_\textbf{m} \cdot (P_2\oplus S_1)}.$ We compute $\xi(\mathds{1}_{C_1}) = \mathds{1}_{Z_1},$ $\xi(\mathds{1}_{C_2}) = \mathds{1}_{Z_2},$ and $\xi(\mathds{1}_{C_3}) = \mathds{1}_{\Lambda_\textbf{m}} + \mathds{1}_{(0,0)} = \mathds{1}_{Z_1} + \mathds{1}_{Z_2}.$ In particular, we see that for $\textbf{m} = (1,2),$ the map \[\xi : \mathcal{H}_\textbf{m} \rightarrow \mathcal{M}_\textbf{m}\] is an isomorphism. When $\textbf{m} = (2,1),$ the kernel of the surjection \[\xi : \mathcal{H}_\textbf{m} \rightarrow \mathcal{M}_\textbf{m}\] is generated by the relation $\mathds{1}_{C_1} - \mathds{1}_{C_3} + \mathds{1}_{C_2}$ for $r(C_1) = (2,1,0),$ $r(C_2) = (0,1,2),$ and $r(C_3) = (1,1,1).$ 

    The indecomposable representations of $\Pi(A_2)$ are of the form $S_1, S_2, Q_1,$ and $Q_2.$ For $\textbf{m} = (1,0),$ $\mathrm{E}_\textbf{m}(\filt) = \mathbb{A}^0 \cup \mathbb{A}^0$ with the two components labelled by $r(C_1) = (1,0,0)$ and $r(C_2) = (0,0,1).$ For $\textbf{m} = (0,1),$ $\mathrm{E}_\textbf{m}(\filt) = \mathbb{A}^0$ with the component labelled by $r(C_3) = (0,1,0).$ For $\textbf{m} = (1,1),$ $\mathrm{E}_\textbf{m}(\filt) = \mathbb{A}^1 \cup \mathbb{A}^1$ with the two components labelled by $r(C_4) = (1,1,0)$ and $r(C_5) = (0,1,1).$  A simple computation gives 
    \[\sum_{C \in \tilde{C}} y^{r(C)}\xi(\mathds{1}_C)(S_1) = y_1+y_3,\]
    \[\sum_{C \in \tilde{C}} y^{r(C)}\xi(\mathds{1}_C)(S_2) = y_2,\]
    \[\sum_{C \in \tilde{C}} y^{r(C)}\xi(\mathds{1}_C)(Q_1) = y_1y_2,\]
    \[\sum_{C \in \tilde{C}} y^{r(C)}\xi(\mathds{1}_C)(Q_2) = y_2y_3.\] In particular, we recover the fact from \cite{GLSmult} that $\varphi_{S_1}\varphi_{S_2} = \varphi_{Q_1}+\varphi_{Q_2}.$    
\end{Example}

\printbibliography


\end{document}